\documentclass[12pt]{article}
\usepackage[utf8]{inputenc}
\usepackage{hyperref}
\usepackage[a4paper,left=3cm,right=3cm,top=2.5cm,bottom=2.5cm]{geometry}
\usepackage{fancyhdr}
\usepackage{setspace}
\usepackage{amsmath,amssymb}
\usepackage{mathtools}
\usepackage{amsfonts}
\usepackage{amsthm}
\usepackage{framed}
\usepackage[all]{xy}
\newtheorem{theorem}{Theorem}[subsection]
\newtheorem{definition}{Definition}[subsection]
\newtheorem{corollary}{Corollary}[subsection]

\title{\textbf{\Large Generalized Karush-Kuhn-Tucker Conditions in Variational and Set-Valued Analysis}}
\author{\textbf{Zhuoyu Xiao\footnote{Department of Mathematics, Jinan University, China. Email: zyxiao@stu2016.jnu.edu.cn}}}
\date{}

\begin{document}
\maketitle
\pagenumbering{roman}

\begin{spacing}{1.2}
\pagestyle{plain}
{\normalsize 
\begin{center}
    \emph{In memory of Jonathan M.Borwein (1951-2016)}
    ~\par \emph{His prodigious contribution changed traditional optimization}
    ~\par \emph{His death is a loss to all those who treasure mathematics}
\end{center}
}

~\par 
\begin{flushleft}
    \textbf{\Large Abstract}
\end{flushleft}
{\normalsize This expository paper contains a concise introduction to some significant works concerning the Karush-Kuhn-Tucker condition, a necessary condition for a solution in local optimality in problems with equality and inequality constraints. The study of this optimality condition has a long history and culminated in the appearance of subdifferentials. The 1970s and early 1980s were important periods for new developments and various generalizations of subdifferentials were introduced, including the Clarke subdifferential and Demyanov-Rubinov quasidifferential. 
\par In this paper, we mainly present four generalized Karush-Kuhn-Tucker conditions or Fritz John conditions in variational analysis and set-valued analysis via Lagrange multiplier methods besides Fr$\Acute{e}$chet differentiable situation, namely subdifferentials of convex functions, generalized gradients of locally Lipschitz functions, quasidifferentials of quasidifferentiable functions and contingent epiderivatives of set-valued maps and discuss the limits of Lagrangian methods slightly in the last chapter. These results represent remarkable developments in the theory of generalized differentiation. The purpose of this paper is to use Karush-Kuhn-Tucker condition as a guide to provide our readers with some advanced topics in modern nonlinear analysis.}

~\par ~\par ~\par
\pagenumbering{arabic}
{\normalsize
\section{Introduction}
\subsection{Some Comments about KKT Condition}
Nonsmooth optimization is among the most difficult tasks in optimization. It
deals with optimization problems that objective and constraint functions are nonsmooth functions. We mainly discuss the following optimization problem from \textsection 3 to \textsection 6 in this paper:
\begin{equation*}
\begin{aligned}
\min \quad & f(x)\\
\textrm{s.t.} \quad & g_{i}(x)\leq 0,\quad (i=1,\dots,m) \\
                    & h_{j}(x) = 0,\quad (j=1,\dots,n)    \\
\end{aligned}
\tag{P1}
\end{equation*}
We follow the same terms like feasible solution, constraint function and optimal solution as other textbooks and work with $\mathbb{R}^n$ space unless otherwise mentioned. In general, there are two different viewpoints of the above problem. One is duality, but this is not our main discussion in this paper. The other viewpoint is optimality conditions including geometric form and Lagrange multiplier type.
\par The development of Lagrange multiplier has a long history. In 1797, Lagrange published his famous multiplier rule \cite{lagrange}, which turned out to be an essential tool in constrained optimization. He applied this principle to infinite dimensional problems in the calculus of variations and then he extended it to finite dimensional optimization problems. It is well known that the Karush-Kuhn-Tucker condition in finite dimensional optimization can be deduced from a general multiplier rule and connect the theories of nonsmooth analysis and optimization.

\subsection{Arrangement of This Paper}
I have tried my best to write this article in a self-contained way. Although in practice we expect a certain mathematical maturity, in principle we assume only knowledge of elementary functional 
analysis. The readers who aren't familiar with functional analysis may refer to \cite{folland2013real} or \cite{rudin1991functional} \nocite{bertsekas2014constrained} \nocite{rockafellar1993lagrange}.
\par Some preliminary knowledge will be introduced in \textsection2. These elementary definitions and theorems will be presented directly, the readers who are not familiar with these materials may refer to \cite{fukushima2011fundamentals}, \cite{hiriart2012fundamentals}, \cite{mordukhovich2013easy}, \cite{clarke2008nonsmooth} and \cite{jahn2010vector}. In \textsection3, we present classical Karush-Kuhn-Tucker condition by variational geometry method under the assumption of Fr$\Acute{e}$chet differentiability. Meanwhile, we will briefly state the relationship between Karush-Kuhn-Tucker condition and Fritz John condition and mention constraint qualifications at the end of this section. Main results of this part are refer to Masao Fukushima's book \cite{fukushima2011fundamentals}.
\par Starting from \textsection4, we turn our attention to the theory of generalized differentiation. In Rockafellar's important  work \cite{rockafellar1970convex}, subdiffential was introduced and deduced necessary condition of convex programming. Next, \textsection5 is devoted to Clarke's work \cite{clarke1990optimization}. In this section, Ekeland variational principle will be presented and will be used in Clarke’s proof of Fritz John condition in Lipschitz optimization. \textsection6 we discuss a totally different class of nonsmooth functions called quasidifferentiable functions whose optimality conditions can be described by subdifferentials and superdifferentials, which is different from the previous types of functions. This part of significant results belongs to Luderer's paper \cite{luderer1991directional}.
\par Although this paper contains no new result, lots of the main theorems and proofs have been simplified, modified and well organized from the original papers and textbooks. It is worth noting that in \textsection7, we discuss an analogous necessary optimality condition characterized by contingent epiderivatives in set-valued optimization, which refer to G{\"o}tz and Johannes's work \cite{gotz2000lagrange}. Although the proof of this key result is a little lengthy, set-valued optimization is a vibrant and promising branch of modern nonlinear analysis. We refer the readers who are interested in set-valued optimization to \cite{jahn2010vector} and \cite{khan2016set} for more details.
\par In \textsection8, we discuss the limits of Lagrangian methods by introducing two pathological examples(the latter one was constructed by the author himself), that is, when the Lagrange multiplier fails. Then we give a more precise claim of classical Fritz John condition, with the necessary assumption of continuity in a neighborhood of the optimal solution. This section refers to Luis A.Fernandez's paper \cite{fernandez1997}.

\section{Some Preliminaries}
\subsection{Functions and Derivatives}
\begin{center}
    \textbf{Lower Semicontinuous Functions}
\end{center}

\begin{definition}
A function $f:\mathbb{R}^n \mapsto \left( -\infty,+\infty \right]$ is lower semicontinuous at $x$ provided that
\begin{equation*}
    \liminf_{x'\to x} f(x') \geq f(x).
\end{equation*}
\end{definition}
\begin{framed}
\noindent \emph{Remark:} This condition is clearly equivalent to saying that for all $\varepsilon > 0$, there exists $\delta > 0$ so that $y \in B(x;\delta)$ implies $f(y) \geq f(x) - \varepsilon$, where as usual, $\infty - r$ is interpreted as $\infty$ when $r \in \mathbb{R}$.
\end{framed}

\begin{definition}
The set defined by a real-valued function $f$ and a real number $\alpha$ as follows:
\begin{equation*}
    S_{f}(\alpha) = \left\{ x\in \mathbb{R}^{n}\:\vert\:f(x) \leq \alpha \right\}
\end{equation*}
is called level set of the function $f$.
\end{definition}

\begin{definition}
The epigraph of  $f:\mathbb{R}^n \mapsto \left( -\infty,+\infty \right]$ is defined by
\begin{equation*}
    epi\,f = \left\{ (x,y) \in \mathbb{R}^n \times \mathbb{R}\:\vert\:x\in dom\,f, y \geq f(x) \right\}.
\end{equation*}
\end{definition}

\begin{theorem}
The following three statements are equivalent:
\begin{itemize}
    \item The function $f:\mathbb{R}^n \mapsto \left( -\infty,+\infty \right]$ is lower semicontinuous.
    \item The level set $S_{f}(\alpha)$ of the function $f$ is a closed set.
    \item The epigraph $epi f$ of the function $f$ is a closed set. 
\end{itemize}
\end{theorem}
\begin{framed}
\noindent \emph{Remark:} The theorem above reveals the equivalence of lower semincontinuity of functions and closeness of corresponding level sets and epigraphs. This approach, considering functions and sets as a whole, is usually a research approach and viewpoint in convex analysis.
\end{framed}

\begin{center}
    \textbf{Classical Derivatives}
\end{center}
\begin{definition}
The directional derivative of f at $x\in dom\,f$ in the direction $d\in \mathbb{R}^n$ is defined as
\begin{equation*}
    f'(x;d) = \lim_{t\to 0^{+}}\frac{f(x+td)-f(x)}{t}
\end{equation*}
when the limit exists. We say that f is G$\hat{\alpha}$teaux differentiable at x provided the limit above exists for all $d\in \mathbb{R}^n$.
\end{definition}
\begin{framed}
\noindent \emph{Remark:} We say $f$ is convex G$\hat{a}$teaux differentiable at $x$ if $f$ is G$\hat{a}$teaux differentiable at $x$ and the function $d \in \mathbb{R}^n\to f'(x;d)\in \mathbb{R}$ is convex or $f$ is linear G$\hat{a}$teaux differentiable at $x$ if $f$ is G$\hat{a}$teaux differentiable at $x$ and the function $d \in \mathbb{R}^n\to f'(x;d)\in \mathbb{R}$ is linear.
\end{framed}

\begin{definition}
Suppose the equality above holds at $x$. We say $f$ is Fr$\Acute{e}$chet differentiable at $x$ if there exists a linear continuous function $f'(x): \mathbb{R}^n \to \mathbb{R}$ such that
\begin{equation*}
    \lim_{\|d\| \to 0} \frac{\|f(x+d)-f(x)-f'(x)\cdot d\|}{\|d\|} = 0,
\end{equation*}
where $f'(x)$ is called the Fr$\Acute{e}$chet derivative of $f(x)$. Usually we also write $f'(x)$ as $\nabla f(x)$.
\end{definition}

\subsection{Basic Properties of Convexity}
\begin{center}
    \textbf{Convex Sets and Support functions}
\end{center}
\begin{definition}
A subset $\Omega\subseteq \mathbb{R}^n$ is convex if the line segment $\left[a,b\right] = \{ \lambda a +(1-\lambda)b\, \vert\, \lambda \in \left[0,1\right]\}$ is entirely contained in $\Omega$ whenever $a,b\in \Omega$.
\end{definition}

\begin{definition}
Given $\alpha_{1},\dots,\alpha_{m}\in \mathbb{R}^n$, the element $x=\sum_{i=1}^m \lambda_{i}\alpha_{i}$, where $\sum_{i=1}^m \lambda_{i}=1$ and $\lambda_{i} \geq 0$ for some $m\in \mathbb{N}$, is called the convex combination of $\alpha_{1},\dots,\alpha_{m}$.
\end{definition}

\begin{definition}
Let $\Omega$ be a subset of $\mathbb{R}^n$. The convex hull of $\Omega$ is defined by
\begin{equation*}
    co\,\Omega = \bigcap\, \left\{ C\, \vert\, C \text{ is convex and }\Omega \subseteq C \right\}.
\end{equation*}
\end{definition}

\begin{theorem}
The convex hull $co\,\Omega$ is the smallest convex set containing $\Omega$. The interior $int\,\Omega$ and the closure $cl\,\Omega$ of a convex set $\Omega$ are also convex.
\end{theorem}

Next we turn our attention to support functions and they play an important role in the proof of optimality condition, as we will see later.

\begin{definition}
Let $S\subseteq \mathbb{R}^n$ be nonempty convex compact set, the support function of $S$ is defined by
\begin{align*}
    \delta_{S}^{\ast}(x) = \max_{s\in S} s^T x,\quad x\in \mathbb{R}^n.
\end{align*}
\end{definition}

\begin{theorem}
Let $S_{1},S_{2} \subseteq \mathbb{R}^n$ be convex compact sets. Then $\delta_{S_{1}}^{\ast}(x) \leq \delta_{S_{2}}^{\ast}(x)$ iff $S_{1}\subseteq S_{2}$.
\end{theorem}

\begin{framed}
\noindent \emph{Remark:} The above theorem can be obtained easily by separation theorem in functional analysis.
\end{framed}

\begin{center}
    \textbf{Convex Functions}
\end{center}
\begin{definition}
Let $f:\Omega \mapsto \left( -\infty,+\infty \right]$ be a real-valued function defined on a convex set $\Omega \subset \mathbb{R}^n$. Then the function f is convex on $\Omega$ if
\begin{equation*}
    f(\lambda x+(1-\lambda)y) \leq \lambda f(x)+(1-\lambda)f(y) \quad \forall x,y \in \Omega\:and\:\lambda \in [0,1].
\end{equation*}
If the inequality is strict for all $x \neq y$, then f is strictly convex on $\Omega$.
\end{definition}

\begin{theorem}
A function $f:\mathbb{R}^n \mapsto \left( -\infty,+\infty \right]$ is convex if and only if its epigraph epi$\,$f is a convex subset of the product space $\mathbb{R}^n \times \mathbb{R}$.
\end{theorem}
\begin{framed}
\noindent \emph{Remark:} The theorem above reveals the equivalence of convexity of functions and convexity of corresponding epigraphs. Here again, we can realize the power of the approach considering functions and sets as a whole mentioned in the remark of Thm 2.1.1.
\end{framed}

\begin{theorem}
Let $f, f_{i}:\mathbb{R}^n \mapsto \left( -\infty,+\infty \right]$ be convex functions for all $i = 1,\dots,m$. Then the following functions are convex as well:
\par
\begin{flushleft}
$\begin{array}{l}{\text { (i) The multiplication by scalars } \lambda f \text { for any } \lambda>0.} \\ {\text { (ii) The sum function } \sum_{i=1}^{m} f_{i}.} \\ {\text { (iii) The maximum function } \max _{1 \leq i \leq m} f_{i}.}\end{array}$
\end{flushleft}
\end{theorem}
\begin{framed}
\noindent \emph{Remark:} In fact, for $i\in I$ be a collection of convex functions with a nonempty index set $I$, the supremum function $f(x) = \sup _{i\in I} f_{i}(x)$ is also convex.
\end{framed}

\subsection{Variational Geometry}
\begin{center}
    \textbf{Cone and Polar Cone}
\end{center}
\begin{definition}
A subset $C\subseteq \mathbb{R}^n$ is called a cone if $x\in C,\alpha \in \left[ 0 , +\infty \right)$ then $\alpha x \in C$. A cone $C$ is called pointed if $C\cap (-C)={0}$. A cone $C$ is called reproducing if $C-C = X$, in this case one also says that $C$ generates $X$. 
\end{definition}

\begin{definition}
The cone generated by a nonempty subset $M$ is denoted
\begin{equation*}
    cone(M) = \{ \lambda x\:\vert\:\lambda \geq 0\:and\: x\in M\}.
\end{equation*}
\end{definition}

\begin{definition}
The polar cone $C^{\ast}$ of any cone $C \subseteq \mathbb{R}^n$ is defined by
\begin{equation*}
    C^{\ast} = \left\{ y\in \mathbb{R}^n\:\vert\:\langle y,x \rangle \leq 0, x \in C \right\}.
\end{equation*}
\end{definition}
\begin{theorem}
We can conclude from the definition above that the polar cones $C^{\ast}$ is a closed convex cone and $C^{\ast} = (co\, C)^{\ast}$. Furthermore, given any two cone $C,D \subseteq \mathbb{R}^n$, if $C \subseteq D$ then $C^{\ast} \supseteq D^{\ast}$.
\end{theorem}

\begin{theorem}
For any nonempty cone $C\subset \mathbb{R}^n$, the polar cone of $C^{\ast}$ namely $C^{\ast \ast}$ is consistent with the closed convex hull of C, that is cl $co\,C$. In particular, if C is a closed convex cone then C = $C^{\ast \ast}$.
\end{theorem}

Finally, an important theorem about convex polyhedral cones will be  introduced. This theorem is essentially equivalent to the Farkas's theorem and will be used in the proof of optimality conditions in \textsection3. The readers can find proof in \cite{fukushima2011fundamentals}.
\begin{theorem}
Consider the closed convex cone generated by vectors $a^{1} ,\dots, a^{m} \in \mathbb{R}^n$ as follows
\begin{equation*}
    C = \left\{ x\in \mathbb{R}^n\:\middle\vert\:x = \sum_{i=1}^m \alpha_{i} a^{i}, \alpha_{i} \geq 0,\quad i = 1,\dots,m \right\},
\end{equation*}
and a closed convex cone composed of all vectors that maintain $90^{\circ}$ or more with each $a^{i}$ vector
\begin{equation*}
    K = \left\{ y\in \mathbb{R}^n \:\middle\vert\: \langle a^{i},y \rangle \leq 0,\quad i = 1,\dots,m\right\},
\end{equation*}
then $K = C^{\ast}$ and $C = K^{\ast}$.
\end{theorem}

\begin{corollary}
Consider the following two closed convex cones defined by $a^1,\dots,a^m \in \mathbb{R}^n$ and $b^1,\dots,b^l \in \mathbb{R}^n$:
\begin{align*}
    C &= \left \{ x\in \mathbb{R}^n\:\middle\vert\:x = \sum\limits_{j=1}^{m} \alpha_{i} a^{i} + \sum\limits_{j=1}^{l} \beta_{j} b^{j},\alpha_{i}\geq 0,i=1,\dots,m,\beta_{j}\in\mathbb{R},j=1,\dots,l \right \},\\
    K &= \left \{y\in\mathbb{R}^n\:\middle\vert\:\langle a^{i},y \rangle \leq 0, i = 1,\dots,m;\langle b^{j},y \rangle = 0,j = 1,\dots,l \right \}.
\end{align*}
Then we have $K = C^{\ast}$ and $C = K^{\ast}$.
\end{corollary}

\begin{center}
    \textbf{Bouligand Tangent Cone and Normal Cone}
\end{center}
Let's consider the geometric concept describing linear approximation of a given set $S\subseteq \mathbb{R}^n$:
\begin{definition}
The Bouligand(or contingent) tangent cone to S at x, denoted $T(S,x)$, is defined as follows:
 \begin{equation*}
    T(S,x) = \left\{ y \in \mathbb{R}^n \:\vert\: \lim_{k\to \infty} \alpha_{k} (x^{k}-x) = y, \lim_{k\to \infty} x^{k} = x, x^{k}\in S, \alpha_{k} \geq 0, k = 1,2,\dots \right\}.
\end{equation*}
\end{definition}

\begin{theorem}
Consider the distance function $d_{S}(x)$ associated with $S$ : $d_{S}(x) = \min \{ \| x - s \| \,\vert\, s\in S \}$. Then we have $v \in T(S,x)$ iff
\begin{equation*}
    \liminf \limits_{t\to 0^{+}} \frac{d_{S}(x+tv)}{t} = 0.
\end{equation*}
\end{theorem}
\begin{framed}
\noindent \emph{Remark:} We can see that the natural concept of tangent cone can be characterized by means of the distance function, sometimes the above formula is also used as an alternative definition of Bouligand tangent cone. Another useful fact is that $T(S,x)$ is always closed for any $S\subseteq \mathbb{R}^n$ and $x\in \mathbb{R}^n$.
\end{framed}

\begin{theorem}
Let $S$ be a nonempty convex set. Then the Bouligand tangent cone $T(S,x)$ is convex for every $x\in S$.
\end{theorem}

\begin{definition}
The polar cone of Bouligand tangent cone $T(S,x)$ is called the normal cone of S at x, denoted by $N(S,x)$.
\end{definition}

\begin{theorem}
More precisely, we often consider the case when S is a convex set. Under this assumption, the normal cone can be expressed as
\begin{equation*}
    N(S,x) = \left\{ z\in \mathbb{R}^n\:\vert\:\langle z,y-x \rangle \leq 0,\quad \forall y \in S \right\}.
\end{equation*}
\end{theorem}

\subsection{Partially Ordered Linear Spaces}
\begin{definition}
Let $X$ be a real linear space. Each nonempty subset $R$ of the product space $X\times X$ is called a binary relation on $X$, we write $xRy$ for $(x,y)\in R$. Every binary relation $\leq$ on $X$ is called a partial ordering on $X$, if the following axioms are satisfied for arbitrary $w,x,y,z\in X:$
\begin{enumerate}
    \item $x\leq x;$
    \item $x\leq y, y\leq z\implies x\leq z;$
    \item $x\leq y, w\leq z\implies x+w\leq y+z;$
    \item $x\leq y,\alpha \in \mathbb{R}_{+}\implies \alpha x\leq \alpha y.$
\end{enumerate}
What's more, a partial ordering $\leq$ on $X$ is called antisymmetric, if the following implication holds for arbitrary $x,y\in X:$
\begin{equation*}
    x\leq y,\:y\leq x \implies x = y.
\end{equation*}
\end{definition}

\begin{definition}
A real linear space equipped with a partial ordering is called partially ordered linear space.
\end{definition}

A significant characterization of a partial ordering in a linear space is given by the following theorem:
\begin{theorem}
Let $X$ be a real linear space. If $C$ is a convex cone in $X$, then the binary relation
\begin{equation*}
    \leq_{C} = \{ (x,y)\in X\times X\:\vert\:y-x\in C \}
\end{equation*}
is a partial ordering on $X$. If, in addition, $C$ is pointed, then $\leq_{C}$ is antisymmetric. 
\end{theorem}

\begin{framed}
\noindent \emph{Remark:} This theorem is easy to prove and is of great importance because a partial ordering can be investigated using convex analysis.
\end{framed}

\begin{definition}
Let $X$ be a real linear space and $X^{\ast}$ denotes the linear space containing all continuous linear functionals on $X$. A convex cone characterizing a partial ordering in $X$ is called an ordering cone and we often denote it by $C_{X}$. Moreover, the dual cone of $C_{X}$ is defined as
\begin{equation*}
    C_{X^{\ast}} = \{x^{\ast} \in X^{\ast}\:\vert\:x^{\ast}(x)\geq 0, \forall x\in C_{X} \}.
\end{equation*}
\end{definition}

\begin{definition}
Let $X$ be a partially ordered linear space and $C_{X}$ is the ordering cone in $X$. For arbitrary elements $x,y\in X$ with $x\leq_{C} y$ the set
\begin{equation*}
    [x,y] = \{ z\in X\:\vert\:x\leq_{C}z\leq_{C}y \} 
\end{equation*}
is called the order interval between $x$ and $y$.
\end{definition}
\begin{framed}
\noindent \emph{Remark:} It's easy to prove that the order interval between $x$ and $y$ can be written as
\begin{equation*}
    [x,y] = (\{x\}+C_{X})\cap (\{y\}-C_{X}).
\end{equation*}
\end{framed}

\subsection{Basic Set-Valued Analysis}
In this part, we begin to make a brief introduction to set-valued analysis including semicontinuity, which will be used in the proof of Ekeland variational principle and Lipschitz optimization in \textsection5. For further properties of set-valued maps, we will present them in \textsection7. This part may refer to Aubin's book  \cite{aubin2009set}. 
\begin{center}
    \textbf{Basic Concepts}
\end{center}

\begin{definition}
Let $X$,$Y$ be real normed spaces. $F$ is called a set-valued map if for any $x\in X$ there exists a corresponding subset $F(x)\subseteq Y$, denoted by $x\rightrightarrows F(x)$ or $F:X\rightrightarrows Y$. The domain and image of $F(x)$ are denoted by $Dom(F)$ and $Im(F)$ respectively:
\begin{equation*}
    \begin{aligned}
    Dom(F) &= \left\{ x\in X \:\middle\vert\: F(x) \neq \varnothing \right\},\\
    Im(F) &= \bigcup_{x\in X}F(x).
    \end{aligned}
\end{equation*}
\end{definition}

\begin{theorem}
Assume that $F, F_{1}, F_{2}$ are set-valued maps from real normed space $X$ to real normed space $Y$ and $\lambda$ be constant. We define $(F_{1}\cap F_{2})$, $(F_{1}\cup F_{2})$, $(F_{1} + F_{2})$ and $\lambda F$ as follows:
\begin{equation*}
    \begin{aligned}
    (F_{1}\cap F_{2})(x) &= F_{1}(x)\cap F_{2}(x),\quad x\in X\\
    (F_{1}\cup F_{2})(x) &= F_{1}(x)\cup F_{2}(x),\quad x\in X\\
    (F_{1} + F_{2})(x) &= F_{1}(x) + F_{2}(x),\quad x\in X\\
    \lambda F(x) &= \left \{ \lambda y\:\middle\vert\: y\in F(x)\right\},\quad x \in X.
    \end{aligned}
\end{equation*}
\end{theorem}

\begin{center}
    \textbf{Semicontinuity of Set-Valued Analysis}
\end{center}
\begin{definition}
(Upper Semicontinuous)
Let $X,Y$ be real normed spaces and $F:X\rightrightarrows Y$ maps from $X$ to $Y$. Given $x_{0}\in Dom(F)$, if for any neighborhood $U$ of $F(x_{0})$ there exists $\delta > 0$ such that
\begin{align*}
    F(x) \subseteq U, \quad \forall x\in B(x_{0},\delta).
\end{align*}
We called that $F$ is upper semicontinuous at $x_{0}$. If $F$ is upper semicontinuous at each point of $Dom(F)$, then $F$ is upper semicontinuous at $X$.
\end{definition}

\begin{definition}
(Lower Semicontinuous)
Let $X,Y$ be real normed spaces. Let $F$ be a set-valued map from $X$ to $Y$. Given $x_{0}\in Dom(F)$, if for any $y_{0}\in F(x_{0})$ and sequence $\{x_{k}\}$ in $Dom(F)$ satisfying $x_{k}\to x_{0}(k\to \infty)$, there exists $\{y_{k}\}\in F(x_{k})$ in $Y$ such that $y_{k}\to y_{0}(k\to \infty)$. We say that $F$ is lower semicontinuous at $x_{0}$. If $F$ is lower semicontinuous at each point of $Dom(F)$, then we say that $F$ is lower  semicontinuous at $X$.
\end{definition}

\begin{theorem}
Let $X,Y$ be real normed spaces and $F$ maps from $X$ to $Y$. Given $x_{0}\in Dom(F)$. If $F(x)$ is compact, upper semicontinuity and lower semicontinuity of set-valued maps can be characterized in following ways:
\begin{itemize}
    \item Upper Semicontinuous: For any $\varepsilon > 0$, there exists a  constant $\delta$ such that
    \begin{align*}
        F(x) \subseteq F(x_{0}) + B(0,\varepsilon),\quad \forall x\in B(x_{0},\delta)
    \end{align*}
    iff $F:X\rightrightarrows Y$ is upper semicontinuous.
    
    \item Lower Semicontinuous: For any $\varepsilon > 0$, there exists a  constant $\delta$ such that
    \begin{align*}
        F(x_{0}) \subseteq F(x) + B(0,\varepsilon),\quad \forall x\in B(x_{0},\delta)
    \end{align*}
    iff $F:X\rightrightarrows Y$ is lower semicontinuous.
\end{itemize}
\end{theorem}

\section{Classical Karush-Kuhn-Tucker Conditions}
In this section, we will present rigorous derivation of Karush-Kuhn-Tucker condition and Fritz John condition using variational geometry method. What's more, a brief introduction of constraint qualifications will be presented in \textsection 3.3. Unless otherwise mentioned, the differentiability here refer to Fr$\Acute{e}$chet  differentiability.

\subsection{Classical KKT Condition}
We first simplify the optimization problem \text{(P1)} mentioned in \textsection 1.1, given function $f:\mathbb{R}^n\mapsto \left( -\infty,+\infty \right]$ and subset $S\subseteq \mathbb{R}^n$, then
\begin{equation*}
\begin{aligned}
\min \quad & f(x)\\
\textrm{s.t.} \quad & x\in S
\end{aligned}
\tag{$\dagger$}
\end{equation*}

\begin{theorem}
Assume $f:\mathbb{R}^n\mapsto \left( -\infty,+\infty \right]$ is differentiable. If f attains local optimal solution of problem $\mathrm{(\dagger)}$ at $x^{\ast}$, then
\begin{equation}
    -\nabla f(x^{\ast}) \in N(S,x^{\ast}).
\end{equation}
\end{theorem}
\begin{proof}
$\forall y\in T(S,x^{\ast})$, it follows from the definition of tangent vector that there exists sequences $\{x^k\}$ and nonnegative numerical sequence $\{\alpha_k\}$ satisfying $\alpha_k(x^k-x^{\ast})\to y$. Since $f$ is Fr$\Acute{e}$chet differentiable, thus from Definition 2.1.5 we have 
\begin{equation}
    f(x^k)-f(x^{\ast}) = \langle \nabla f(x^{\ast}),x^k-x^{\ast} \rangle + o(\parallel x^k-x^{\ast} \parallel).
\end{equation}
Note that $f$ attains its local optimal solution at $x^{\ast}$, thus $f(x^k) \geq f(x^{\ast})$ for sufficiently large $k$. By equality\textrm{(2)} above, we have
\begin{equation*}
    \langle \nabla f(x^{\ast}),\alpha_k(x^k-x^{\ast}) \rangle + \frac{o(\parallel x^k-x^{\ast} \parallel)}{\parallel x^k-x^{\ast} \parallel}\times \alpha_k \parallel x^k-x^{\ast} \parallel\:\geq 0.
\end{equation*}
Let $k\to \infty$ so that $\langle \nabla f(x^{\ast}),y \rangle + 0 \times \parallel y \parallel\:\geq 0$, namely $\langle -\nabla f(x^{\ast}),y \rangle \leq 0$. It's easy to obtain that $-\nabla f(x^{\ast})\in N(S,x^{\ast})$ since $y\in T(S,x^{\ast})$. 
\end{proof}

When the feasible region of problem \textrm{($\dagger$)} is expressed by a collection of functions $g_{i}:\mathbb{R}^n\mapsto \mathbb{R}$ as follows:
\begin{equation}
    S = \{ x\in \mathbb{R}^n\:\vert\:g_{i}(x) \leq 0,\quad i = 1,\dots,m \},
\end{equation}
then problem \textrm{($\dagger$)} can be written as
\begin{equation*}
\begin{aligned}
\min \quad & f(x)\\
\textrm{s.t.} \quad & g_{i}(x) \leq 0, \quad i = 1,\dots,m
\end{aligned}
\tag{P2}
\end{equation*}
The constraints satisfying $g_{i}(x) = 0$ are called active constraints at $x$, and corresponding index set denoted by $I(x) = \{i\:\vert\:g_{i}(x) = 0\} \subseteq \{1,\dots,m\}$. In \textsection 2, we define the linear approximation of $S$ at $x$ called Bouligand tangent cone, here we again define another linear approximation of $S$.
\begin{definition}
Under the assumption that each $g_{i}(x)$ is Fr$\Acute{e}$chet differentiable at $x$ and $S$ can be expressed in the formula\textrm{(3)}. The cone 
\begin{equation}
    C(S,x) = \left\{ y\in \mathbb{R}^n\:\vert\: \langle \nabla g_{i}(x),y \rangle \leq 0,\quad i\in I(x)\right\}
\end{equation}
is called linearizing cone of $S$ at $x$.
\end{definition}
\begin{framed}
\noindent \emph{Remark:} It's a  fact that $T(S,x)\subseteq C(S,x)$ always holds but not vice versa, which can be refered to Masao Fukushima's book \cite{fukushima2011fundamentals}.
\end{framed}

\begin{theorem}
\textbf{(KKT Condition)} Assume that $x^{\ast}$ is a local optimal solution of problem $\mathrm{(P2)}$, objective function $f:\mathbb{R}^n \mapsto \mathbb{R}$ and constraint functions $g_{i}:\mathbb{R}^n \mapsto \mathbb{R}(i = 1,\dots,m)$ are all differentiable at $x^{\ast}$. If $C(S,x^{\ast})\subseteq co\,T(S,x^{\ast})$ holds, then exists $\overline{\lambda} \in \mathbb{R}^m$ satisfying 
\begin{equation}
    \begin{aligned}
    &\nabla f(x^{\ast})+\sum_{i=1}^m \overline{\lambda_{i}} \nabla g_{i}(x^{\ast}) = 0,\\
    &\overline{\lambda_{i}} \geq 0,\:\overline{\lambda_{i}}g_{i}(x^{\ast}) = 0,\: i = 1,\dots,m
    \end{aligned}
\end{equation}
\end{theorem}
\begin{proof}
Since $x^{\ast}$ is the local optimal solution of problem \textrm{(P2)}, thus $-\nabla f(x^{\ast})\in N(S,x^{\ast})$ according to Theorem 3.1.1. We deduce from Theorem 2.3.1 that 
\begin{align*}
    C(S,x^{\ast})^{\ast} \supseteq co\,T(S,x^{\ast})^{\ast} = T(S,x^{\ast})^{\ast} = N(S,x^{\ast}),
\end{align*}
hence $-\nabla f(x^{\ast})\in C(S,x^{\ast})^{\ast}$. It follows from the definition of $C(S,x^{\ast})$ and Theorem 2.3.3 there exists $\overline{\lambda_{i}}\geq 0(i\in I(x^{\ast}))$ satisfying
\begin{align*}
    -\nabla f(x^{\ast}) = \sum\limits_{i\in I(x^{\ast})} \overline{\lambda_{i}}\nabla g_{i}(x^{\ast}).
\end{align*}
Let $\overline{\lambda_{i}} = 0(i\notin I(x^{\ast}))$ then yield the desired result.
\end{proof}

Now we consider problem \textrm{(P1)}. The following corollary is a generalization of Theorem 3.1.2 under additional equality constraints. The index set is also defined by $I(x) = \{i\:\vert\:g_{i}(x) = 0\} \subseteq \{1,\dots,m\}$. Now we define feasible region $S$ as follows:
\begin{equation*}
    S = \left\{ x\in \mathbb{R}^n \:\middle\vert\: g_{i}(x) \leq 0,i = 1,\dots,m, h_{j}(x) = 0,j = 1,\dots,l\right\}.
\end{equation*}
The tangent cone of $S$ at $x$ is denoted by $T(S,x)$, and linearizing cone $C(S,x)$ can be expressed as
{\normalsize \begin{equation*}
    C(S,x) = \left\{ y\in \mathbb{R}^n \:\middle\vert\: \langle\nabla g_{i}(x),y \rangle \leq 0,i\in I(x),\langle \nabla h_{j}(x),y \rangle = 0,j = 1,\dots,l \right\}.
\end{equation*}}

\begin{corollary}
\textbf{(KKT Condition)} Let $x^{\ast}$ be a local optimal of problem $\mathrm{(P1)}$, objective function $f:\mathbb{R}^n\mapsto \mathbb{R}$ and constraint functions $g_{i}:\mathbb{R}^n\mapsto \mathbb{R}(i = 1,\dots,m)$, $h_{j}:\mathbb{R}^n\mapsto \mathbb{R}(j = 1,\dots,n)$ are all differentiable at $x^{\ast}$. If $C(S,x^{\ast})\subseteq co\,T(S,x^{\ast})$, there exists $\overline{\lambda}\in \mathbb{R}^m$, $\overline{\mu}\in \mathbb{R}^n$ satisfying
\begin{equation}
\begin{aligned}
&\nabla f(x^{\ast}) + \sum_{i=1}^m \overline{\lambda_{i}} \nabla g_{i}(x^{\ast}) + \sum_{j=1}^n \overline{\mu_{j}} \nabla h_{j}(x^{\ast}) = 0,\\
&\overline{\lambda_{i}} \geq 0,\: \overline{\lambda_{i}}g_{i}(x^{\ast}) = 0,\: i = 1,\dots,m
\end{aligned}
\end{equation}
\end{corollary}
\begin{proof}
It is not difficult to see that $-\nabla f(x^{\ast})\in C(S,x^{\ast})^{\ast}$ also holds if $C(S,x^{\ast})\subseteq co\,T(S,x^{\ast})$. It follows from Corollary 2.3.1 that there exists $\lambda_{i} \geq 0 (i\in I(x))$ and $\overline{\mu_{j}} (j = 1,\dots,n)$ satisfying
\begin{equation*}
    -\nabla f(x^{\ast}) = \sum_{i=1}^m \overline{\lambda_{i}} \nabla g_{i}(x^{\ast}) + \sum_{j=1}^n \overline{\mu_{j}} \nabla h_{j}(x^{\ast})
\end{equation*}
for those $i\notin I(x^{\ast})$ let $\overline{\lambda_{i}} = 0$ hence establishes the desired result.
\end{proof}

\subsection{Classical Fritz John Condition}
In the proof of classical Karush-Kuhn-Tucker condition, we note that  the condition $C(S,x)\subseteq co\,T(S,x)$ must be satisfied, which is called constraint qualification in constrainted optimization and will be discussed later. In this part, Fritz John condition will be obtained directly without any constraint qualification.

\begin{theorem}
\textbf{(Fritz John Condition)} Let $x^{\ast}$ be a local optimal solution of problem $\mathrm{(P2)}$, objective function $f:\mathbb{R}^n \mapsto \mathbb{R}$ and constraint functions $g_{i}:\mathbb{R}^n \mapsto \mathbb{R}(i = 1,\dots,m)$ are all differentiable at $x^{\ast}$. There exist $\overline{\lambda_{0}}, \overline{\lambda_{1}}, \dots, \overline{\lambda_{m}}$ s.t. $\sum_{i=0}^m \overline{\lambda_{i}} = 1$ satisfying
\begin{equation}
\begin{aligned}
&\overline{\lambda_{0}}\nabla f(x^{\ast}) + \sum_{i=1}^m \overline{\lambda_{i}} \nabla g_{i}(x^{\ast}) = 0,\\
&\overline{\lambda_{i}} \geq 0,\: \overline{\lambda_{i}}g_{i}(x^{\ast}) = 0,\: i = 1,\dots,m 
\end{aligned}
\end{equation}
\end{theorem}
\begin{proof}
Define a set as follows:\emph{(Note that $y$ is a vector in $\mathbb{R}^n$)}
\begin{equation*}
    Y = \left\{ y\in \mathbb{R}^n \:\middle\vert\: \langle \nabla f(x^{\ast}),y \rangle < 0,\langle \nabla g_{i}(x^{\ast}), y \rangle < 0\:(i\in I(x^{\ast})) \right\}.
\end{equation*}
It follows that $Y$ is empty. In fact, if there exists $y\in Y$, then it's easy to prove that both $f(x^{\ast}+\alpha y) < f(x^{\ast})$ and $g_{i}(x^{\ast}+\alpha y) < 0\:(i = 1,\dots,m)$ hold for sufficiently small $\alpha > 0$, which contradicts the fact that $x^{\ast}$ is a local optimal solution. Now we define a convex cone $C\subseteq \mathbb{R}^n\,(i\in I(x^{\ast}))$
\begin{align*}
    C &= \left \{(y_{0},y)^T\in \mathbb{R}^{n+1} \:\middle\vert\: y_{0} + \langle \nabla f(x^{\ast}),y \rangle \leq 0, y_{0} + \langle \nabla g_{i}(x^{\ast}),y \rangle \leq 0\: \right \} \\
    &= \left \{(y_{0},y)^T\in \mathbb{R}^{n+1} \:\middle\vert\: \langle(1,\nabla f(x^{\ast}))^T, (y_{0},y)^T\rangle \leq 0, \langle(1,\nabla g_{i}(x^{\ast}))^T, (y_{0},y)^T\rangle \leq 0 \right \}.
\end{align*}
\par \noindent Since $Y$ is empty, it follows that $y_{0}\leq 0$ for any $(y_{0},y)^T\in C$. We conclude that $(1,0)^T\in C^{\ast}$, which is proven by calculating $\langle (1,0)^T,(y_{0},y)^T\rangle = y_{0} \leq 0,\:\forall\:(y_{0},y)^T\in C$. It is clear from Theorem 2.3.3 that there exists nonnegative $\overline{\lambda_{i}}(i\in{0}\cup I(x^{\ast}))$ satisfying
\begin{equation*}
\begin{aligned}
&\overline{\lambda_{0}}\nabla f(x^{\ast}) + \sum_{i\in I(x^{\ast})}\overline{\lambda_{i}} \nabla g_{i}(x^{\ast}) = 0,\\
&\overline{\lambda_{0}} + \sum_{i\in I(x^{\ast})}\overline{\lambda_{i}} = 1, \:i\in I(x^{\ast})
\end{aligned}
\end{equation*}
Let $\overline{\lambda_{i}} = 0$ when $i\notin I(x^{\ast})$, completing the proof of the theorem.
\end{proof}

\subsection{Constraint Qualifications}
\begin{center}
\textbf{Some Comments about KKT and Fritz John Conditions}
\end{center}
From the theorem discussed in \textsection3.2, we can see that Fritz John condition still holds although $C(S,x)\subseteq co\,T(S,x)$ does not hold. When $\overline{\lambda_{0}} = 0$, Fritz John condition doesn't contain any information about the objective function $f(x)$, which is a pathological phenomenon. Only under the condition of constraint qualifications, we can assure that $\overline{\lambda_{0}} > 0$. Then Fritz John condition is reasonable and equivalent to KKT condition (Divided by $\overline{\lambda_{0}}$ and replace $\overline{\lambda_{i}}$ with $\overline{\lambda_{i}}/\overline{\lambda_{0}}$).
\begin{center}
\textbf{Constraint Qualifications}
\end{center}
In this part, we will present following constraint qualifications related to problem \textrm{(P1)} under the assumption of Fr$\Acute{e}$chet differentiability and make a brief discussion about relationship between them.
\begin{itemize}
    \item \textbf{Linear Independence Constraint Qualification:} 
    \par $h_{j}(j = 1,\dots,n)$ are continuously differentiable at $x$, and $\nabla g_{i}(x)(i\in I(x)), \nabla h_{j}(x)\newline(j = 1,\dots,n)$ are linearly independent.
    
    \item \textbf{Slater's Constraint Qualification:}
    \par $g_{i}(i\in I(x))$ are convex functions, and $h_{j}(j = 1,\dots,n)$ are affine functions (that is, $h(x)=\langle x,\zeta \rangle + \beta$ for $\zeta \in \mathbb{R}^n$ and $\beta \in \mathbb{R}$), and exists $x^{0}$, such that $g_{i}(x^{0}) < 0(i = 1,\dots,m)$ and $h_j(x^{0}) = 0(j = 1,\dots,n)$.
    
    \item \textbf{Mangasarian-Fromovitz Constraint Qualification:}
    \par $h_{j}(j = 1,\dots,n)$ are continuously differentiable at $x$ and $\nabla h_{j}(j = 1,\dots,n)$ are linearly independent. There exists $y\in \mathbb{R}^n$, such that $\langle \nabla g_{i}(x),y \rangle < 0(i\in I(x))$ and $\langle \nabla h_{j}(x),y \rangle = 0(j = 1,\dots,n)$.
    
    \item \textbf{Abadie Constraint Qualification:} $C(S,x)\subseteq T(S,x)$.
    
    \item \textbf{Guignard Constraint Qualification:} $C(S,x)\subseteq co\,T(S,x)$.
\end{itemize}

\begin{theorem}
The figure below reveals the relationship between the above constraint qualifications.
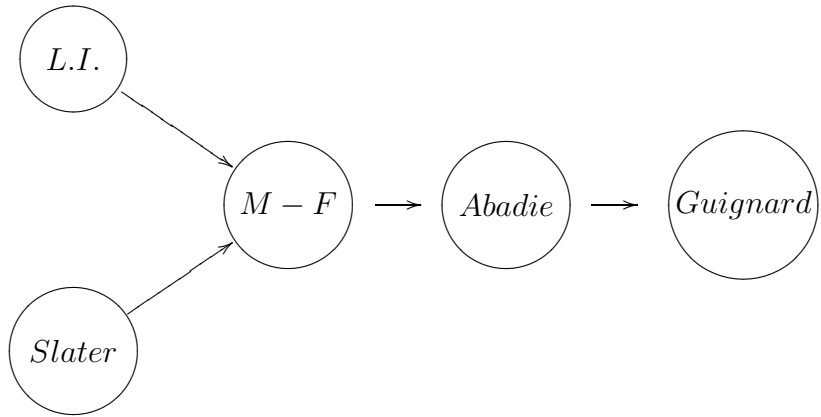
\[ \xymatrix@=0.6cm{
*+<1cm>[o][F]{L.I.} \ar@{->}[rd]&  &  &  \\
 & *+<1cm>[o][F]{M-F}\ar@{->}[r] & *+<1cm>[o][F]{Abadie}\ar@{->}[r] & *+<1cm>[o][F]{Guignard}\\
*+<1cm>[o][F]{Slater}\ar@{->}[ur] 
} \]
\end{theorem}
\begin{framed}
\noindent \emph{Remark:} Here we give the result directly without detailed proof since constraint qualifications are not main topics in this paper, the readers who take interest in these materials may refer to \cite{fukushima2011fundamentals}.
\end{framed}

\section{Convex Programming}
\subsection{Introduction to Subdifferentials}
From geometric viewpoint, a function $f:\mathbb{R}^n \mapsto \mathbb{R}$ is convex if and only if its tangent line is below the graph. The concept of subdifferential of convex functions can be introduced based on this property.
\begin{definition}
Let $f(x)$ be a convex function on $\mathbb{R}^n$, the subdifferential of $f(x)$ at $x$ denoted by $\partial f(x)$, defined as follows:
\begin{align}
    \partial f(x) = \left \{ \xi \in \mathbb{R}^n \:\middle\vert\: f(y) \geq f(x) + \xi^T(y-x), y\in \mathbb{R}^n \right \},
\end{align}
$\xi \in \partial f(x)$ is the element of subdifferential, called subgradient.
\end{definition}
\begin{framed}
\noindent \emph{Remark:} It's easy to verify that $\partial f(x)$ is a closed convex set from the definition.
\end{framed}

\begin{corollary}
From the definition of directional derivative, it is not difficult to conclude that $\xi \in \partial f(x)$ iff
\begin{align}
    f'(x;d)\geq \xi^T d,\quad \forall d\in\mathbb{R}^n.
\end{align}
\end{corollary}

Now we present two theorems describing the subdifferential of supremum function $f(x) = \sup \left \{ f_{i}(x) \:\middle\vert\: i\in T\right \}$ under the assumption that $T $ is a index set and $f_{i}(x)\:(i\in T)$ are convex functions on $\mathbb{R}^n$. It follows from the remark of Theorem 2.2.4 that $f(x)$ is also a convex function.

\begin{theorem}
The subdifferential $\partial f(x)$ of supremum function $f(x)$ satisfying the following:
\begin{align*}
    cl\,co\,\bigcup_{i\in T(x)} \partial f_{i}(x) \subseteq \partial f(x)
\end{align*}
where $T(x) = \left \{ i\in T\:\middle\vert\: f_{i}(x) = f(x) \right \}$.
\end{theorem}
\begin{proof}
Given $i\in T(x)$ and $\xi \in \partial f_{i}(x)$, it follows from the definition of subdifferential that
\begin{equation*}
    \begin{aligned}
    f(y) &\geq f_{i}(y) \geq f_{i}(x) + \xi^T(y-x)\\
         &= f(x) + \xi^T(y-x),\quad \forall y\in \mathbb{R}^n.
    \end{aligned}
\end{equation*}
which implies $\xi \in \partial f(x)$, hence
\begin{align*}
    \partial f_{i}(x) \subseteq \partial f(x),\quad i\in T(x)
\end{align*}
Since $\partial f(x)$ is close and convex, we obtain the inclusion relationship as required. 
\end{proof}

This theorem only illustrates the inclusion relationship on one side, the next theorem states that the equation holds under certain conditions.
\begin{theorem}
Assume that $f(x)$ is the supremum function of a collection convex  functions $f_{i}(x)\,(i\in T)$. Given $x\in \mathbb{R}^n$, we define function $h_{i}:i\to f_{i}(x)$. Under the condition that $T$ is compact and $h_{i}$ is upper semicontinuous, we have
\begin{align}
    co\,\bigcup_{i\in T(x)} \partial f_{i}(x) = \partial f(x).
\end{align}
\end{theorem}
\begin{framed}
\noindent \emph{Remark:} The proof of this theorem involves many lemmas hence we omit details and use it directly. The reader who wants to acquire detailed proof may refer to \cite{mordukhovich2013easy}.
\end{framed}

\subsection{Lagrangian Methods for Convex Propramming}
Before our discussion of Fritz John condition, we first introduce a useful theorem called extreme condition.
\begin{theorem}
Let $f(x)$ be a convex function on $\mathbb{R}^n$, then $x^{\ast}$ is the minimum point of $f(x)$ iff $0\in \partial f(x^{\ast})$.
\end{theorem}
\begin{proof}
Assume that $0\in \partial f(x^{\ast})$, according the definition of subdifferential, for any $x\in \mathbb{R}^n$ we have
\begin{align*}
    f(x) - f(x^{\ast}) \geq 0^T(x-x^{\ast}),
\end{align*}
thus $f(x) \geq f(x^{\ast})\:(\forall x\in \mathbb{R}^n)$, which implies $f(x)$ attains minimum at $x^{\ast}$. On the other hand, let $x^{\ast}$ be minimum point of $f(x)$ then $f(x)\geq f(x^{\ast})\:(\forall x\in \mathbb{R}^n)$, that is,
\begin{align*}
    f(x) \geq f(x^{\ast}) + 0^T (x-x^{\ast}),
\end{align*}
hence $0\in \partial f(x^{\ast})$ from the definition of subdifferential, completing the proof.
\end{proof}

For simplicity, we first discuss generalized Fritz John condition of problem $\textrm{(P2)}$ in the following theorem.
\begin{theorem}
\textbf{(Fritz John Condition)} Let $f(x),g_{i}(x)\,(i = 1,\dots,m)$ are all convex functions in problem $\mathrm{(P2)}$ and $f(x)$ attains minimum at $x^{\ast}$, then exists a sequence $\overline{\lambda_{i}}\,(i = 0,\dots,m)$ s.t. $\sum_{i=0}^m \overline{\lambda_{i}} = 1$ satisfying
\begin{equation}
    \begin{aligned}
    & 0\in \overline{\lambda_{0}} \partial f(x^{\ast}) + \sum_{i=1}^m \overline{\lambda_{i}} \partial g_{i}(x^{\ast}),\\
    & \overline{\lambda_{i}}g_{i}(x^{\ast}) = 0,\quad i = 1,\dots,m.
    \end{aligned}
\end{equation}
\end{theorem}
\begin{proof}
Define the following function:
\begin{align}
    F(x) = \max \left \{ f(x)-f(x^{\ast}), g_{1}(x), \dots, g_m(x) \right \}.
\end{align}
It is easy to verify that $F(x) \geq 0, \forall x\in \mathbb{R}^n, F(x^{\ast}) = 0$, thus $F(x)$ attains its minimum at $x^{\ast}$. It follows from Theorem 2.2.4 that $F(x)$ is convex, which shows that $0\in \partial F(x^{\ast})$ according to Theorem 4.2.1. Applying Theorem 4.1.2, we have
\begin{align}
    \partial F(x^{\ast}) = co\,\left \{ \partial f(x^{\ast}) \:\bigcup\: (\bigcup\limits_{i\in I(x^{\ast})} \partial g_{i}(x^{\ast})) \right \},
\end{align}
where $I(x^{\ast}) = \left \{ i\in {1,\dots,m} \:\vert\: g_{i}(x^{\ast}) = 0 \right \}$. Note that the right side of (14) can be expressed as 
\begin{align}
    \left \{ \overline{\lambda_{0}} \xi + \sum_{i\in I(x^{\ast})} \overline{\lambda_{i}} \xi_{i} \:\middle\vert\: \xi \in \partial f(x^{\ast}),\xi_{i}\in \partial g_{i}(x^{\ast}),\overline{\lambda_{i}}\geq 0,i\in I(x^{\ast}), \sum_{i\in{0}\cup I(x^{\ast})}\overline{\lambda_{i}} = 1 \right \}.
\end{align}
It follows that 
\begin{equation}
    \begin{aligned}
    0 \in \overline{\lambda_{0}}\partial f(x^{\ast}) + \sum_{i\in I(x^{\ast})}\overline{\lambda_{i}} \partial g_{i}(x^{\ast}),\\
    \overline{\lambda_{i}}\geq 0,i\in I(x^{\ast}), \sum_{i\in{0}\cup I(x^{\ast})}\overline{\lambda_{i}} = 1
    \end{aligned}
\end{equation}
Let 
\begin{align}
    \overline{\lambda_{i}} = 0,\quad i\in \{1,\dots,m\}\setminus I(x^{\ast}).
\end{align}
Then we obtain the desired result.
\end{proof}

\begin{theorem}
\textbf{(Fritz John Condition)} Let $x^{\ast}$ be a local optimal solution of problem $\mathrm{(P1)}$, objective function $f:\mathbb{R}^n\mapsto \mathbb{R}$ and constraint functions $g_{i}:\mathbb{R}^n\mapsto \mathbb{R}(i=1,\dots,m)$ and $h_{j}:\mathbb{R}^n\mapsto \mathbb{R}(j=1,\dots,n)$ are all convex functions. Then exist $\overline{\lambda_{0}},\overline{\lambda_{1}},\dots,\overline{\lambda_{m}}$ and $\overline{\mu_{1}},\dots,\overline{\mu_{n}}$ satisfying
\begin{equation}
    \begin{aligned}
    & 0\in \overline{\lambda_{0}} \partial f(x^{\ast}) + \sum_{i=1}^m \overline{\lambda_{i}} \partial g_{i}(x^{\ast}) + \sum_{j=1}^n \overline{\mu_{j}} \partial h_{j}(x^{\ast}),\\
    & \overline{\lambda_{i}}g_{i}(x^{\ast}) = 0,\quad i = 1,\dots,m.
    \end{aligned}
\end{equation}
\end{theorem}
\begin{framed}
\noindent \emph{Remark:} This theorem can be found in Rockafellar's book \cite{rockafellar1970convex} and the proof of it concerns the concept of saddle point thus we present it directly.
\end{framed}

\section{Locally Lipschitz Programming}
\subsection{Introduction to Generalized Gradients}
\begin{definition}
Let $f:\mathbb{R}^n \mapsto \mathbb{R}$ be Lipschitz of rank K near a given point $x\in X$; that is, for some $\varepsilon > 0$, we have
\begin{align*}
    \left| f(y) - f(z) \right| \leq K \parallel y - z \parallel. \quad \forall y,z \in B(x,\varepsilon).
\end{align*}
\end{definition}
\begin{framed}
\noindent \emph{Remark:} It's easy to verify that $\left| d_{S}(y) - d_{S}(z) \right| \leq \parallel y - z \parallel$ for any convex compact set $S\subseteq \mathbb{R}^n$, which implies the rank of distance function $d_{S}(x)$ is 1.
\end{framed}

\begin{definition}
The generalized directional derivative of f at x in the direction v, denoted by $f^{\circ}(x;v)$, is defined as follows:
\begin{align*}
    f^{\circ}(x ; v) =\limsup _{y \rightarrow x,\:t\to 0^{+}} \frac{f(y+t v)-f(y)}{t}
\end{align*}
where of course $y$ is a vector in $X$ and $t$ is a positive scalar.
\end{definition}

\begin{definition}
Let $f(x)$ be a locally Lipschitz function on $\mathbb{R}^n$, generalized gradient of $f(x)$ denoted by $\partial^{\circ} f(x)$, is defined as follows:
\begin{align*}
    \partial^{\circ} f(x) = \left \{ \xi\in\mathbb{R}^n \:\middle\vert\: f^{\circ}(x ; v) \geq \xi^T v,\quad \forall v\in \mathbb{R}^n \right \}.
\end{align*}
For distinction, sometimes generalized directional derivative and generalized gradient are called Clarke directional derivative and Clarke subdifferential respectively in the literature.
\end{definition}

We will introduce some useful properties of Clarke subdifferential, which will be used in the proof of generalized Fritz John condition in locally Lipschitz optimization.
\begin{theorem}
Let $f(x)$ be a locally Lipschitz function on $\mathbb{R}^n$ of rank $L$ at $x$, then $\partial^{\circ} f(x)$ is convex and compact and $\parallel \xi \parallel \leq L,\,\forall \xi \in \partial^{\circ} f(x)$. That is, $\partial^{\circ} f(x) \subseteq cl\,B(0,L)$. In particular, For any convex and compact set $S\subseteq \mathbb{R}^n$, we have the inclusion relation $\partial^{\circ} d_{S}(x) \subseteq cl\,B(0,1)$.
\end{theorem}

\begin{theorem}
Let $f(x)$ and $g(x)$ be locally Lipschitz functions on $\mathbb{R}^n$, $c$ is a given constant. Then we have
\begin{flushleft}
(i) $\partial^{\circ} (cf(x)) = c \partial^{\circ} f(x)$.\\
(ii) $\partial^{\circ} (f(x) + g(x)) \subseteq \partial^{\circ} f(x) + \partial^{\circ} g(x)$.
\end{flushleft}
\end{theorem}

In Theorem 4.2.1 we discuss extreme condition of convex function via subdifferential, which has the similar form under the condition of locally Lipschitz function, as you can see in the next theorem:
\begin{theorem}
Let $f(x)$ be a locally Lipschitz function on $\mathbb{R}^n$. If $f(x)$ attains its minimum or maximum at $x^{\ast}$, then we have $0\in \partial^{\circ} f(x^{\ast})$. 
\end{theorem}
\begin{proof}
Since $\partial(-f(x))=-\partial f(x)$, we only need to consider one situation. Suppose $x^{\ast}$ is the minimum point, thus
\begin{equation*}
    \begin{aligned}
    f^{\circ}(x^{\ast} ; v) &=\limsup _{y \rightarrow x^{\ast},\:t\to 0^{+}} \frac{f(y+t v)-f(y)}{t}\\
    &\geq \limsup_{t\to 0^{+}} \frac{f(x^{\ast} + tv) - f(x^{\ast})}{t}\\
    &\geq 0,\quad \forall v\in \mathbb{R}^n
    \end{aligned}
\end{equation*}
which implies that $0\in \partial^{\circ} f(x^{\ast})$ from the difinition of Clarke subdifferential.
\end{proof}

\begin{theorem}
Let $I$ be a finite set and for all $i\in I$ let $f_{i}(x)$ be locally Lipschitz functions around $x^{\ast}$. Then the function $f(x) = \max \{ f_{i}(x)\:\vert\: i\in I \}$ satisfies
\begin{align*}
    \partial^{\circ} f(x^{\ast}) \subseteq co\,\{ \partial^{\circ} f_{i}(x^{\ast}) \:\vert\: i\in I(x^{\ast}) \},
\end{align*}
where $I(x^{\ast}) = \{i\in I \:\vert\: f_{i}(x^{\ast}) = f(x^{\ast}) \}$.
\end{theorem}
\begin{framed}
\noindent \emph{Remark:} Compared with Theorem 4.1.1, the inclusion relation of Clarke subdifferential of supremum function is different from that of the subdifferential of supremum function .
\end{framed}

\subsection{Ekeland Variational Principle}
In this part, we only focus on Ekeland variational principle which holds in any complete metric space. Roughly speaking, a variational principle asserts that, for any lower semicontinuous function which is bounded below, one can add a small perturbation to make it attain a minimum. In fact, there are many other variational principles in modern variational analysis. The reader who shows an interest in them may refer to Borwein's book \cite{borwein2004techniques}.
\begin{theorem}
\textbf{(Ekeland Variational Principle)} Let $(X,d)$ be a complete metric space and let $f:X\mapsto \left( -\infty,+\infty \right]$ be a lower semicontinuous function bounded from below. Suppose that $\varepsilon > 0$ and $z\in X$ satisfy
\begin{align*}
    f(z) < \inf_{x\in X} f(x) + \varepsilon.
\end{align*}
Then there exists $y\in X$ such that
\begin{flushleft}
$(i)\:d(z,y) \leq 1$,\\
$(ii)\:f(y)+\varepsilon d(z,y) \leq f(z)$,\\
$(iii)\:f(x)+\varepsilon d(x,y) \geq f(y)\:(\forall x\in X)$.
\end{flushleft}
\end{theorem}
\begin{proof}
Define a sequence ($z_{i}$) by induction starting with $z_{0}=z$. Suppose that we have defined $z_{i}$. Set
\begin{align*}
    S_{i} = \left \{ x\in X \:\middle\vert\: f(x)+\varepsilon d(x,z_{i}) \leq f(z_{i}) \right \}
\end{align*}
and consider two possible cases:(a) $\inf_{S_{i}} f = f(z_{i})$. Then we define $z_{i+1} = z_{i}$.(b) $\inf_{S_{i}} f < f(z_{i})$. We choose $z_{i+1}\in S_{i}$ such that
\begin{align}
    f(z_{i+1}) < \inf_{S_{i}} f + \frac{1}{2}[f(z_{i})-\inf_{S_{i}} f] = \frac{1}{2}[f(z_{i})+\inf_{S_{i}} f] < f(z_{i}).
\end{align}
We show that $(z_{i})$ is a Cauchy sequence. In fact, if (a) ever happens then $z_{i}$ is stationary for $i$ large. Otherwise,
\begin{align}
    \varepsilon d(z_{i},z_{i+1}) \leq f(z_{i})-f(z_{i+1}).
\end{align}
Adding (20) up from $i$ to $j-1>i$ we have
\begin{align}
    \varepsilon d(z_{i},z_{j}) \leq f(z_{i})-f(z_{j}).
\end{align}
Observe that the sequence $(f(z_{i}))$ is decreasing and bounded from below by $\inf_{X} f$, and therefore convergent. We conclude from (21) that $(z_{i})$ is Cauchy. Let $y = \lim_{i\to \infty}z_{i}$. We show that $y$ satisfies the conclusions of the theorem. Setting $i=0$ in (21) we have
\begin{align}
    \varepsilon d(z,z_{j}) + f(z_{j}) \leq f(z).
\end{align}
Taking limits as $j\to \infty$ yields (ii). Since $f(z)-f(y) \leq f(z) - \inf_{X} f < \varepsilon$, (i) follows from (ii). It remains to show that $y$ satisfies (iii). Fixing $i$ in (21) and taking limits as $j\to \infty$ yields $y\in S_{i}$. That is to say
\begin{align*}
    y \in \bigcap_{i=1}^{\infty} S_{i}.
\end{align*}
On the other hand, if $x\in \cap_{i=1}^{\infty} S_{i}$ then, for all $i = 1,2,\dots,$
\begin{align}
    \varepsilon d(x,z_{i+1}) \leq f(z_{i+1}) - f(x) \leq f(z_{i+1}) - \inf_{S_{i}} f.
\end{align}
It follows that from (19) that $f(z_{i+1})-\inf_{S_{i}} f \leq f(z_{i})-f(z_{i+1})$, and therefore $\lim_{i}[f(z_{i+1})-\inf_{S_{i}}f] = 0$. Taking limits in (23) as $i\to \infty$ we have $\varepsilon d(x,y) = 0$. It follows that
\begin{align}
    \bigcap_{i=1}^{\infty} S_{i} = \{ y \}.
\end{align}
Notice that the sequence of sets $(S_{i})$ is nested, i.e., for any $i$, $S_{i+1} \subseteq S_{i}$. In fact, for any $x\in S_{i+1}$, $f(x)+\varepsilon d(x,z_{i+1})\leq f(z_{i+1})$ and $z_{i+1}\in S_{i}$ yields
\begin{equation}
    \begin{aligned}
    f(x) + \varepsilon d(x,z_{i}) &\leq f(x) + \varepsilon d(x,z_{i+1}) + \varepsilon d(z_{i},z_{i+1})\\
    &\leq f(z_{i+1}) + \varepsilon d(z_{i},z_{i+1}) \leq f(z_{i}).
    \end{aligned}
\end{equation}
which implies that $x\in S_{i}$. Now, for any $x\neq y$, it follows from (24) that when $i$ sufficiently large $x\notin S_{i}$. Thus, $f(x)+\varepsilon d(x,z_{i})\geq f(z_{i})$. Taking limits as $i\to \infty$ we arrive at (iii).
\end{proof}

\begin{corollary}
Let $(X,d)$ be a complete metric space and let $f:X\mapsto \left( -\infty,+\infty \right]$ be a lower semicontinous function bounded from below. Suppose that $\varepsilon > 0$ and $z\in X$ satisfy
\begin{align*}
    f(z) < \inf_{x\in X} f(x) + \varepsilon.
\end{align*}
Then for any $\lambda > 0$ there exists $y\in X$ such that
\begin{flushleft}
$(i)\:d(z,y) \leq \lambda$,\\
$(ii)\:f(y)+(\varepsilon/\lambda) d(z,y) \leq f(z)$,\\
$(iii)\:f(x)+(\varepsilon/\lambda) d(x,y) \geq f(y)\:(\forall x\in X\setminus \{y\})$.
\end{flushleft}
\end{corollary}

\subsection{Lagrangian Methods for Locally Lipschitz Programming}
In this section, generalized Fritz John condition of locally Lipschitz optimization will be discussed and presented. Necessary optimality condition of problem \textrm{(P2)} will be obtained easily as the proof in Theorem 4.2.2. 
\par We turn our attention mainly to probelm \textrm{(P1)} and we will see the power of Ekeland Variational Principle in the proof. The part of work is devoted to Francis H.Clarke \cite{clarke1990optimization}.
\begin{theorem}
\textbf{(Fritz John Condition)} Let $f(x), g_{i}(x)\,(i = 1,\dots,m)$ are all locally Lipschitz functions on $\mathbb{R}^n$. Assume that $f(x)$ attains its minimum at $x^{\ast}$ of problem $\mathrm{(P2)}$, then exists a sequence $\overline{\lambda_{i}}\,(i = 0,1,\dots,m)$ satisfying
\begin{equation}
    \begin{aligned}
    & 0\in \overline{\lambda_{0}} \partial^{\circ} f(x^{\ast}) + \sum_{i=1}^m \overline{\lambda_{i}} \partial^{\circ}\partial g_{i}(x^{\ast}),\\
    & \overline{\lambda_{i}}g_{i}(x^{\ast}) = 0,\quad i = 1,\dots,m.
    \end{aligned}
\end{equation}
\end{theorem}
\begin{proof}
Construct the following function
\begin{align*}
    F(x) = \max \left \{ f(x)-f(x^{\ast}),g_1(x),\dots,g_m(x) \right \}.
\end{align*}
Applying Theorem 5.1.3 and Theorem 5.1.4, together with the proof of Theorem 4.2.2, it is easy to acquire the required result.
\end{proof}

\begin{theorem}
\textbf{(Fritz John Condition)} Let $f(x), g_{i}(x)\,(i = 1,\dots,m),\\ h_{j}(x)\,(j = 1,\dots,n)$ are locally Lipschitz functions. Assume that $x^{\ast}$ is the minimum point of problem $\mathrm{(P1)}$, then exist $\overline{\lambda_{i}}\geq 0,i = 0,1,\dots,m$ and $\overline{\mu_{j}},j = 1,\dots,n$ such that
\begin{equation}
    \begin{aligned}
    & 0\in \overline{\lambda_{0}}\partial^{\circ} f(x^{\ast}) + \sum_{i=1}^m \overline{\lambda_{i}} \partial^{\circ} g_{i}(x^{\ast}) + \sum_{j=1}^n \overline{\mu_{j}} \partial^{\circ} h_{j}(x^{\ast}),\\
    & \overline{\lambda_{i}} g_{i}(x^{\ast}) = 0,\quad i = 1,\dots,m
    \end{aligned}
\end{equation}
\end{theorem}
\begin{proof}
Given $\varepsilon > 0$, we define $T$ and $F(x)$ as follows:
\begin{align*}
    T = \left \{ t = (\overline{\lambda_{0}}, \overline{\lambda}^T, \overline{\mu}^T)^T\in \mathbb{R}^{1+m+n} \:\middle\vert\: \overline{\lambda_{0}},\overline{\lambda} \geq 0,\parallel (\overline{\lambda_{0}},\overline{\lambda},\overline{\mu}) \parallel = 1 \right \},\\
    F(x) = \max_{(\overline{\lambda_{0}},\overline{\lambda}^T,\overline{\mu}^T)^T\in T} \left \{ \overline{\lambda_{0}}(f(x)-f(x^{\ast})+\varepsilon) + \overline{\lambda}^T g(x) + \overline{\mu}^T h(x)\right \},
\end{align*}
where $\overline{\lambda_{0}}\in \mathbb{R},\overline{\lambda} \in \mathbb{R}^m,\overline{\mu} \in \mathbb{R}^n$,
\begin{equation*}
    \begin{aligned}
    g(x) = (g_{1}(x),\dots,g_{m}(x))^T,\\
    h(x) = (h_{1}(x),\dots,h_{n}(x))^T,
    \end{aligned}
\end{equation*}
It is easy to proof that $F(x)$ is Lipschitz at $x^{\ast}$ and $F(x^{\ast})=\varepsilon$. On the other hand, we have $F(x)>0,x\in \mathbb{R}^n$. If not, there exists $y\in \mathbb{R}^n$ such that $F(y)\leq 0$, which implies
\begin{align*}
    g_{i}(y)\leq 0, h_{j}(y)\leq 0, f(y)\leq f(x^{\ast})-\varepsilon.
\end{align*}
This contradicts the fact that $x^{\ast}$ is the minimum point. Hence,
\begin{align*}
    F(x^{\ast})\leq \inf_{x\in \mathbb{R}^n} F(x) + \varepsilon.
\end{align*}
It follows from Corollary 5.2.1 that there exists $u \in B(x^{\ast},\sqrt{\varepsilon})$ such that
\begin{align*}
    F(u)\leq F(x) + \sqrt{\varepsilon} \parallel x - u \parallel
\end{align*}
for any $x\in \mathbb{R}^n$, which implies $F(x) + \sqrt{\varepsilon}\parallel x - u \parallel$ attains its minimum at $x = u$. Applying Theorem 5.1.1 and the fact that $\partial^{\circ} \parallel x - u \parallel \subseteq \partial^{\circ} d_{\{u\}}(x) \subseteq cl\,B(0,1)$, we have the fact that
\begin{align*}
    0\in \partial^{\circ} F(u) + cl\,B(0,\sqrt{\varepsilon}).
\end{align*}
We proceed to proof that the set-valued map 
\begin{align*}
    (t,x)\to \partial^{\circ} L(x,t)
\end{align*}
is upper semicontinuous, where $t = (\overline{\lambda_{0}},\overline{\lambda}^T,\overline{\mu}^T)^T$ and
\begin{align*}
    L(x,t) = \overline{\lambda_{0}} f(x) + \overline{\lambda}^T g(x) + \overline{\mu}^T h(x).
\end{align*}
Note that $\forall t_1,t_2\in T$, the following 
\begin{align*}
    x \to L(x,t_1)-L(x,t_2) = (t_1-t_2)^T (f(x),g(x),h(x))
\end{align*}
is a Lipschitz function and $L \parallel t_1 - t_2 \parallel$ is Lipschitz constant, where $L = \max \{ L_{f}, L_{g}, L_{h} \}$, thus
\begin{align*}
    \partial^{\circ} L(x,t_{1}) \subseteq \partial^{\circ} L(x,t_{2}) + L\parallel t_{1} - t_{2} \parallel cl\,B(0,1),
\end{align*}
which implies set-valued map $(t,x)\to \partial^{\circ} L(x,t)$ is upper semicontinuous. Since $F(u)>0$, then there exists unique $t_{u}\in T$ such that $F(x)$ attains its maximum at $t_{u}$, hence
\begin{align}
    0\in \partial^{\circ} L(u,t_u) + cl\,B(0,\sqrt{\varepsilon}).
\end{align}
Note that the $i\,th$ element of $\overline{\lambda}$, namely $\overline{\lambda}_{i}$ equals 0 if $g_{i}(u)<0$.  Taking limits $\varepsilon_{i} \to 0$, then we have $u_{i}\to x^{\ast}$ and there exists a subsequence of $\{t_{u_{i}}\}$ converging to some element in $T$. We now combine (28) with the upper semicontinuity of set-valued maps $(t,x)\to \partial^{\circ} L(x,t)$ to conclude the required result.
\end{proof}

\section{Quasidifferentiable Programming}
In this section, several concepts and properties of quasidifferentiable functions will be presented first\nocite{demianov1995constructive}. Then we will only focus on Fritz John condition in quasidifferentiable optimization with inequality constraints. This important result belongs to Luderer's paper \cite{luderer1991directional}. The readers wants to find more results in the case with equality and
inequality constraints may refer to paper \cite{gao2000demyanov}.
\subsection{Introduction to Quasidifferentials}
\begin{definition}
Let $f(x)$ be directionally differentiable at $x$ and there exist a pair of convex compact sets $\underline{\partial} f(x), \overline{\partial} f(x) \subseteq \mathbb{R}^n$ such that
\begin{align*}
     f'(x;d) = \max_{u \in \underline{\partial} f(x)} u^T d + \min_{v \in \overline{\partial} f(x)} v^T d,\quad \forall d \in \mathbb{R}^n
\end{align*}
We call $f(x)$ is quasidifferentiable at $x$ and $D f(x) = [\underline{\partial} f(x),  \overline{\partial} f(x)]$ is the quasidifferential of $f(x)$. $\underline{\partial} f(x)$ and $\overline{\partial} f(x)$ are called the subdifferential and superdifferential of $f(x)$, respectively. What's more, we call $f(x)$ subdifferentiable if $\overline{\partial} f(x) = \{0\}$ and $f(x)$ superdifferentible if $\underline{\partial} f(x) = \{0\}$.
\end{definition}

Before our discussion of quasisubdifferentials of quasidifferentiable functions, we first define addition and scalar multiplication of set pairs.
\begin{definition}
Let $U_1$, $V_1$, $U_2$, $V_2$ $\subseteq \mathbb{R}^n$ and $c$ is a constant, the addition and scalar multiplication of set pair $[U_1,\:V_1]$ and $[U_2,\:V_2]$ are defined as follows:
\begin{equation*}
    \begin{aligned}
        &[U_1,\:V_1] + [U_2,\:V_2] = [U_1 + U_2,\:V_1 + V_2],\\
        &c [U_1,\:U_2] = 
        \begin{cases}
        [c\,U_1,\:c\,U_2],\quad c \geq 0.\\
        [c\,U_2,\:c\,U_1],\quad c < 0.
        \end{cases}
    \end{aligned}
\end{equation*}
\end{definition}

\begin{theorem}
Let $f_{1}(x)$ and $f_{2}(x)$ be quasidifferentiable functions on $\mathbb{R}^n$, then $f_{1}(x)+f_{2}(x)$, $f_{1}(x)f_{2}(x)$ and $cf_{1}(x)\,(c\in \mathbb{R})$ are all quasidifferentiable functions. If $f_{1}(x) \neq 0$, then $\frac{1}{f_{1}(x)}$ is also quasidifferentiable. Furthermore, we have the following rules
\par
\begin{equation*}
    \begin{aligned}
    &D(f_{1}(x)+f_{2}(x)) = Df_{1}(x) + Df_{2}(x),\\
    &D(f_{1}(x)f_{2}(x)) = f_{1}(x) Df_{2}(x) + f_{2}(x) Df_{1}(x),\\
    &D(cf_{1}(x)) = c Df_{1}(x),\\
    &D(\frac{1}{f_{1}(x)}) = -\frac{1}{f_{1}^2(x)} Df_{1}(x).
    \end{aligned}
\end{equation*}
\end{theorem}

As in \textsection4 and \textsection5, we present explicit expression of the quasidifferential of maximum function $f(x) = \max_{1\leq i\leq m} f_{i}(x)$.
\begin{theorem}
Assume that $f_{i}(x),i = 1,\dots,m$ are quasidifferentiable functions on $\mathbb{R}^n$, then the maximum function $f(x) = \max_{1\leq i\leq m} f_{i}(x)$ is also quasidifferentiable. Its quasidifferential $[\underline{\partial} f(x),  \overline{\partial} f(x)]$ can be expressed as follows:
\begin{equation}
    \begin{aligned}
    &\underline{\partial} f(x) = co\bigcup_{k\in I(x)}( \underline{\partial} f_{k}(x) - \sum_{i\in I(x)\setminus \{ k \} } \overline{\partial} f_{i}(x) ),\\
    &\overline{\partial} f(x) = \sum_{i\in I(x)} \overline{\partial} f_{i}(x).
    \end{aligned}
\end{equation}
where $I(x) = \left \{ i\in \{1,\dots,m\} \:\middle\vert\: f_{i}(x) = f(x) \right \}$.
\end{theorem}

\begin{framed}
\noindent \emph{Remark:} This theorem illustrates explicit expressions of subdifferential and superdifferential of maximum function $f(x)=\max_{i\leq i\leq m} f_{i}(x)$. The proof of this theorem is a little complicate thus we omit it, the readers who are interested in it may refer to \cite{GaoY}.
\end{framed}

\begin{theorem}
Let $f(x)$ be a quasidifferentiable function on $\mathbb{R}^n$. If $f$ attains its minimum at $x^{\ast}$, then
\begin{align}
    -\overline{\partial} f(x^{\ast}) \subseteq \underline{\partial} f(x^{\ast}).
\end{align}
\end{theorem}
\begin{proof}
Since $f(x)$ is directionally differentiable and attains its minimum at  $x^{\ast}$, thus we have
\begin{align}
    f'(x^{\ast};d) \geq 0,\quad \forall d\in \mathbb{R}^n.
\end{align}
If not, then exists $d_{1}\in \mathbb{R}^n$ such that $f'(x^{\ast};d_1) < 0$. It follows from the definition of directional derivative that $f(x^{\ast}+td_{1}) < f(x^{\ast})$ for sufficiently small $t>0$, which contradicts the fact that $x^{\ast}$ is the minimum point of $f$.
Combining (31) and the definition of quasidifferential, we derive
\begin{align*}
    0 \leq \max_{u\in \underline{\partial} f(x^{\ast})} u^T d + \min_{v\in \overline{\partial} f(x^{\ast})} v^T d,\quad \forall d\in \mathbb{R}^n,
\end{align*}
that is
\begin{align*}
    \max_{v\in -\overline{\partial} f(x^{\ast})} v^T d \leq \max_{u\in \underline{\partial} f(x^{\ast})} u^T d,\quad \forall d\in \mathbb{R}^n,
\end{align*}
which implies the fact that
\begin{align*}
    \delta_{-\overline{\partial} f(x^{\ast})}^{\ast}(x) \leq \delta_{\underline{\partial} f(x^{\ast})}^{\ast}(x).
\end{align*}
According Theorem 2.2.2, we obtain $-\overline{\partial} f(x^{\ast}) \subseteq \underline{\partial} f(x^{\ast})$, as required.
\end{proof}

\subsection{Larangian Methods for Quasidifferentiable Programming}
In this section, we will discuss and present generalized Fritz John condition in quasidifferentiable mathematical programming problem \textrm{(P2)}. For the convenience of expression, we write $g_{i}(x)\,(i=1,\dots,m)$ to $f_{i}(x)\,(i=1,\dots,m)$ and write $f(x)$ to $f_{0}(x)$. That is
\begin{equation*}
\begin{aligned}
\min \quad & f_{0}(x)\\
\textrm{s.t.} \quad & f_{i}(x) \leq 0, \quad i = 1,\dots,m
\end{aligned}
\tag{P2}
\end{equation*}

\begin{theorem}
Let $f_{i}(x)(i = 0,\dots,m)$ are all quasidifferentiable functions on $\mathbb{R}^n$. Assume that $f_{0}(x)$ attains its minimum at $x^{\ast}$ of problem $\mathrm{(P2)}$, then we have
\begin{align}
    -\sum_{i\in \{0\}\cup I(x^{\ast})}\overline{\partial} f_{i}(x^{\ast}) \subseteq co\,\bigcup_{i\in \{0\} \cup I(x^{\ast})} \left \{ \underline{\partial}f_{i}(x^{\ast}) - \sum_{j\in \{0\} \cup I(x^{\ast})\setminus \{i\} } \overline{\partial} f_{j}(x^{\ast})\right \},
\end{align}
where $I(x^{\ast}) = \left \{ i\in \{1,\dots,m\} \:\vert\: f_{i}(x^{\ast}) = 0 \right \}$.
\end{theorem}
\begin{proof}
Let $F(x) = \max \{ f_{0}(x)-f_{0}(x^{\ast}), f_{1}(x), \dots, f_{m}(x) \}$. Note that $F(x)$ is also quasidifferentiable and $F(x^{\ast}) = 0$. Besides, we have $F(x)\geq F(x^{\ast})$ for $x$  sufficiently close to $x^{\ast}$, that is, $F$ attains its minimum at $x^{\ast}$. From Theorem 6.1.3, we derive
\begin{align}
    -\overline{\partial} F(x^{\ast}) \subseteq \underline{\partial} F(x^{\ast}).
\end{align}
Applying Theorem 6.1.2, $-\overline{\partial} F(x^{\ast})$ equals the left side and $\underline{\partial} F(x^{\ast})$ equals the right side. This completes the proof.
\end{proof}
\begin{framed}
    \noindent \emph{Remark:} Compared with the optimality conditions of  convex programming and Lipschitz programming presented in \textsection4 and \textsection5, we can see that of quasidifferentiable programming has something different. It seems that the result above has nothing to do with Lagrange multipliers. The good news is that we can change it to a familiar form. In the following two theorems, we will present the main results of Luderer's paper \cite{luderer1991directional}.
\end{framed}

Before our discussion of the quasidifferentiable case of problem $\mathrm{(P2)}$, we firstly deal with the subdifferentiable case, that is, the objective function and constraints functions are only subdifferentiable (see the Defintion 6.1.1). The proof of next theorem concerns a lot of literature and requires a certain mathematical maturity thus we present and use it directly, the reader may refer to \cite{luderer1991directional} for further reading.
\newline
\begin{theorem}
\textbf{(Fritz John Condition)} Let $f_{i}(x)(i = 0,\dots,m)$ are all subdifferentiable functions on $\mathbb{R}^n$. Assume that $f_{0}(x)$ attains its minimum at $x^{\ast}$ of problem $\mathrm{(P2)}$. Then there exist scalars $\overline{\lambda_{i}} \geq 0,\:(i = 0,\dots,m)$ such that
\begin{equation}
    \begin{aligned}
    & 0\in \sum_{i=0}^m \overline{\lambda_{i}}\:\underline{\partial} f_{i}(x^{\ast}),\\
    & \overline{\lambda_{i}}f_{i}(x^{\ast}) = 0,\quad i = 1,\dots,m,
    \end{aligned}
\end{equation}
If, in addition, there exists a vector $\hat{x}$ with
\begin{equation*}
    f'_{i}(x^{\ast}; \hat{x}-x^{\ast})<0,\quad \forall i\in I(\overline{x}),
\end{equation*}
(generalized Slater condition), then we have the fact that $\overline{\lambda_{0}}\neq 0$.
\end{theorem}

\begin{framed}
\noindent \emph{Remark:} Note that the necessary condition above is in accordance with the well-known Lagrange multiplier principle. In the general case, when we deal with a quasidifferentiable problem $\mathrm{(P2)}$, the direct Lagrange principle fails. Instead, we are able to state a so-called weakened Lagrange multiplier principle. In turn, this leads to the following result.
\end{framed}

\begin{theorem}
\textbf{(Fritz John Condition)} Let $f_{i}(x)(i=0,\dots,m)$ are all quasidifferentiable functions on $\mathbb{R}^n$ with the quasidifferentials $Df_{i}(x^{\ast}) = [\underline{\partial} f_{i}(x^{\ast}),\overline{\partial} f_{i}(x^{\ast})],(i=0,\dots,m)$. Assume that $f_{0}(x)$ attains its minimum at $x^{\ast}$ of problem $\mathrm{(P2)}$. Then, for any $w_{i}\in \overline{\partial} f_{i}(x^{\ast})$, $i\in {0}\cup I(x^{\ast})$ there exist scalars $\overline{\lambda_{i}} \geq 0\:(i = 0,\dots,m)$ not all zero, such that
\begin{equation}
    \begin{aligned}
    & 0\in \sum_{i\in \{0\}\cup I(x^{\ast})} \overline{\lambda_{i}}\:(\underline{\partial}  f_{i}(x^{\ast})+w_{i}),\\
    & \overline{\lambda_{i}}f_{i}(x^{\ast}) = 0,\quad i = 1,\dots,m,
    \end{aligned}
\end{equation}
If, in addition, the regularity condition, that is, there exists $\hat{r}$ such that 
\begin{equation}
    \max_{z\in \underline{\partial} f_{i}(x^{\ast})} \langle z,\hat{r} \rangle + \max_{z\in \overline{\partial} f_{i}(x^{\ast})} \langle z,\hat{r} \rangle < 0,\quad \forall i \in I(x^{\ast})
    \tag{RC}
\end{equation}
is satisfied, then actually $\overline{\lambda_{0}}\neq 0$ and this theorem becomes extended Karush-Kuhn-Tucker theorem.
\end{theorem}

\begin{proof}
Fix $w_{i}\in \overline{\partial} f_{i}(x^{\ast})$, $i\in {0}\cup I(x^{\ast})$ and let $f_{0,w_{0}}(x)$ (analogously $f_{i,{w_{i}}}(x)$) be a function associated with $x^{\ast}$, defined via the relation:
\begin{equation*}
    f_{0,{w_{0}}}(x) = f_{0}(x^{\ast}) + \max \{\langle z,x-x^{\ast} \rangle \vert z\in \underline{\partial} f_{0}(x^{\ast}) + w_{0}\}
\end{equation*}
and having the properties $f_{0,w_{0}}(x^{\ast})=f_{0}(x^{\ast})$,$Df_{0,w_{0}}(x^{\ast})=[\underline{\partial}f_{0}(x^{\ast})+w_{0},{0}]$ which implies that $f_{0,w_{0}}(x)$ is subdifferential function, and $f'_{0}(x^{\ast};r) = \min \{ f'_{0,w_{0}}(x^{\ast};r) \vert w_{0}\in \overline{\partial} f_{0}(x^{\ast}) \}$. It's easy to see that, at the point $x^{\ast}$, there cannot exist a direction $\overline{r}$ satisfying simultaneously the conditions $f'_{0,w_{0}}(x^{\ast};\overline{r})<0$ and $f'_{i,w_{i}}(x^{\ast};\overline{r})<0$ for $i\in I(x^{\ast})$. In fact, if we could indicate such a direction, then, by what was said above, $f'_{0}(x^{\ast};\overline{r})<0$, $f'_{i}(x^{\ast};\overline{r})<0$ for $i\in I(x^{\ast})$. This, however, contradicts the assumption that $x^{\ast}$ provides a local minimum in problem $\mathrm{(P2)}$. Thus, considering the subdifferentiable problem as follow:
\begin{equation*}
    \begin{aligned}
    \min \quad & f_{0,w_{0}}(x)\\
    \textrm{s.t.} \quad & f_{i,w_{i}}(x) \leq 0, \quad i\in I(x^{\ast})
    \end{aligned}
\end{equation*}
Since the objective function and constraint functions are subdifferentiable functions and $x^{\ast}$ is the minimum solution, we conclude the existence of multipliers $\overline{\lambda_{i}}\:(i\in \{0\} \cup I(x^{\ast}))$ satisfying Thm 6.2.2, that is (35). Finally, taking the fixed elements $w_{i}\in \overline{\partial} f_{i}(x^{\ast})\:(i\in I(x^{\ast}))$, condition (RC) guarantees at $x^{\ast}$ the validity of the generalized Slater condition for every function $f_{i,w_{i}}$, which in turn ensures $\overline{\lambda_{0}}\neq 0$.
\end{proof}

\section{Set-Valued Optimization}
Throughout this section we will use the following standard assumption.
\newline
\newline
\noindent \textbf{Assumption:} \emph{Let $\left(X,\parallel \cdot \parallel_{X} \right)$ be a real normed space, let $\left(Y,\parallel \cdot \parallel_{Y} \right)$ and $\left(Z,\parallel \cdot \parallel_{Z} \right)$ be real normed spaces and partially ordered by convex pointed cones $C_Y \subseteq Y$ and $C_Z \subseteq Z$ respectively, let $\hat{S}$ be a nonempty subset of $X$, and let $F:\hat{S}\rightrightarrows Y$ and $G:\hat{S}\rightrightarrows Z$ be set-valued maps.}
\par Under this assumption we consider the following constrained set-valued optimization problem:
\begin{equation*}
\begin{aligned}
\min \quad & F(x)\\
\textrm{s.t.} \quad & G(x)\cap (-C_{Z}) \neq \emptyset 
\end{aligned}
\tag{P3}
\end{equation*}
For simplicity let $S = \left \{ x\in \hat{S} \:\vert\: G(x)\cap (-C_{Z})\neq \emptyset \right\}$ denote the feasible set of this problem, which is assumed to be nonempty.

\subsection{Some Preliminaries in Set-Valued Optimization}
\begin{definition}
Let the problem $\mathrm{(P3)}$ be given. Let $F(S) = \cup_{x\in S} F(x)$ denote the image set of $F$.
\begin{itemize}
    \item A pair $(x^{\ast},y^{\ast})$ with $x^{\ast}\in S$ and $y^{\ast}\in F(x^{\ast})$ is called a minimizer of the problem $\mathrm{(P3)}$, if $\overline{y}$ is a minimal element of the set $F(S)$, i.e.,
    \begin{equation*}
        (\{y^{\ast}\} - C_{Y})\cap F(S) \subseteq  \{y^{\ast}\} + C_{Y}.
    \end{equation*}
    
    \item A pair $(x^{\ast},y^{\ast})$ with $x^{\ast}\in S$ and $y^{\ast}\in F(x^{\ast})$ is called a strong  minimizer of the problem $\mathrm{(P3)}$, if $y^{\ast}$ is a strongly minimal element of the set $F(S)$, i.e.,
    \begin{equation*}
        F(S) \subseteq \{y^{\ast}\} + C_{Y}.
    \end{equation*}
    
    \item A pair $(x^{\ast},y^{\ast})$ with $x^{\ast}\in S$ and $y^{\ast}\in F(x^{\ast})$ is called a weak  minimizer of the problem $\mathrm{(P3)}$, if $y^{\ast}$ is a weakly minimal element of the set $F(S)$, i.e.,
    \begin{equation*}
        (\{ y^{\ast}\} - int(C_{Y}))\cap F(S) = \emptyset.
    \end{equation*}
\end{itemize}
\end{definition}

\begin{theorem}
Let the assumption mentioned above holds. Then every strong minimizer of the problem $\mathrm{(P3)}$ is also a minimizer of the problem $\mathrm{(P3)}$ and every minimizer of the problem $\mathrm{(P3)}$ is also a weak minimizer of the problem $\mathrm{(P3)}$.
\end{theorem}

\begin{framed}
\noindent \emph{Remark:} This theorem describes the relation between different optimality notions. We omit the proofs and the readers can find them in Jahannes Jahn's book \cite{jahn2010vector}.
\end{framed}

\subsection{Contingent Epiderivatives of Set-Valued Mappings}
\begin{definition}
Let the assumption mentioned in the beginning holds. In addition, let $\hat{S}$ be convex. The set-valued map $F:\hat{S}\rightrightarrows Y$ is called $C_{Y}$-convex, if for all $x_{1},x_{2}\in \hat{S}$ and $\lambda\in [0,1]$
\begin{equation*}
    \lambda F(x_1) + (1-\lambda) F(x_2) \subseteq F(\lambda x_1 + (1-\lambda) x_2) + C_{Y}.
\end{equation*}
\end{definition}

Theorem 2.2.3 describes the equivalence of convexity of a function and convexity of its epigraph. This result also holds in set-valued analysis. $C$-convexity of a set-valued map can also be characterized by the convexity of its epigraph. We present the definition of epigraph of a set-valued map and then proof this characterization. 

\begin{definition}
Let the assumption mentioned above be satisfied. In addition, let $\hat{S}$ be convex. The set
\begin{equation*}
    epi(F) = \{(x,y)\in X\times Y\:\vert\:x\in \hat{S}, y\in F(x)+C_{Y}\}
\end{equation*}
is called the epigraph of $F$.
\end{definition}

\begin{theorem}
Let the assumption mentioned above be satisfied. In addition, let $\hat{S}$ be convex. Then $F$ is $C_{Y}$-convex if and only if $epi(F)$ is a convex set.
\end{theorem}
\begin{proof}
$(\Leftarrow)$ Let $F$ be $C_{Y}$-convex. Take arbitrary elements $(x_1,y_1), (x_2,y_2)\in epi(F)$ and $\lambda \in [0,1]$. Because of the convexity of $\hat{S}$ we have
\begin{equation*}
    \lambda x_1 + (1-\lambda) x_2 \in \hat{S},
\end{equation*}
and since $F$ is $C_{Y}$-convex, we obtain
\begin{align*}
    \lambda y_1 + (1-\lambda) y_2 &\in \lambda (F(x_1) + C_{Y}) + (1-\lambda) (F(x_2) + C_{Y})\\
    &= \lambda F(x_1) + (1-\lambda)F(x_2) + C_{Y}\\
    &\subseteq F(\lambda x_1 + (1-\lambda) x_2) + C_{Y},
\end{align*}
which implies
\begin{equation*}
    \lambda (x_1,y_1) + (1-\lambda) (x_2,y_2) \in epi(F).
\end{equation*}
Consequently, $epi(F)$ is a convex set.
\newline
$(\Rightarrow)$ On the other hand, now we assume that $epi(F)$ is a convex set. Let $x_1,x_2\in \hat{S}, y_1\in F(x_1), y_2\in F(x_2)$ and $\lambda \in [0,1]$ be arbitrary given. Because of the convexity of $epi(F)$ we obtain
\begin{equation*}
    \lambda (x_1,y_1) + (1-\lambda)(x_2,y_2)\in epi(F)
\end{equation*}
implying
\begin{equation*}
    \lambda y_1 + (1-\lambda)y_2 \in F(\lambda x_1+(1-\lambda) x_2) + C_{Y}.
\end{equation*}
Hence, $F$ is $C_{Y}$-convex.
\end{proof}

\begin{definition}
Let the assumption mentioned above be satisfied. Let a pair $(x^{\ast},y^{\ast})$ with $x^{\ast}\in S$ and $y^{\ast}\in F(x^{\ast})$ be given. A single-valued map $DF(x^{\ast},y^{\ast}):X\to Y$ whose epigraph equals the contingent cone to the epigraph of $F$ at $(x^{\ast},y^{\ast})$, i.e.,
\begin{equation*}
    epi(DF(x^{\ast},y^{\ast}) = T(epi(F),(x^{\ast},y^{\ast})),
\end{equation*}
is called contingent epiderivative of $F$ at $(x^{\ast},y^{\ast})$.
\end{definition}

\begin{theorem}
Let the assumption mentioned above be satisfied, and, in addition, let $C_{Y}$ be pointed, let $\hat{S}$ be convex, and let $F$ be $C_{Y}$-convex. If the contingent epiderivative $DF(x^{\ast},y^{\ast})$ of $F$ at $(x^{\ast},y^{\ast})$ exists, then it is sublinear, namely positive homogeneous and subadditive.
\end{theorem}

\subsection{Lagrangian Methods in Set-Valued Optimization}
\begin{theorem}
\textbf{(Fritz John Condition)} Let the cone $C_Y$ have a nonempty interior int$\left( C_{Y} \right)$, let the set $\hat{S}$ be convex, and let the maps $F$ and $G$ be $C_{Y}$-convex and $C_{Z}$-convex, respectively. Assume that $\left( x^{\ast},y^{\ast} \right) \in X\times Y$ with $x^{\ast}\in S$ and $y^{\ast}\in F(x^{\ast})$ is a weak minimizer of the problem $\mathrm{(P3)}$. Let the contingent epiderivative of $\left( F,G \right)$ at $\left( x^{\ast},\left( y^{\ast}, z^{\ast} \right) \right)$ for an arbitrary $z^{\ast}\in G(x^{\ast})\cap \left( -C_{Z} \right)$ exist. Then there are continuous linear functionals $t\in C_{Y^{\ast}}$ and $u\in C_{Z^{\ast}}$ with $\left( t,u \right)\neq \left( 0_{Y^{\ast}},0_{Z^{\ast}} \right)$ so that
\begin{equation*}
    t(y) + u(z) \geq 0\:for\:all\:(y,z) = D(F,G)(x^{\ast},(y^{\ast},z^{\ast}))(x- x^{\ast}) \:\:with\:\: x\in \hat{S}
\end{equation*}
and
\begin{equation*}
    u(z^{\ast})=0.
\end{equation*}
If in addition to the above assumptions, the regularity assumption
\begin{equation}
    \{ z\:\vert\: (y,z)\in D(F,G)(x^{\ast},(y^{\ast},z^{\ast}))(cone(S-\{x^{\ast}\}))\} + cone(C_{Z}+\{z^{\ast}\}) = Z
\end{equation}
is satisfied, then $t\neq 0_{Y^{\ast}}$.
\end{theorem}
\begin{proof}
In the product space $Y\times Z$ we define for an arbitrary $z^{\ast}\in G(x^{\ast})\cap (-C_{Z})$ the following set:
\begin{equation*}
    M = \left[ \bigcup_{x\in \hat{S}} D(F,G)(x^{\ast},(y^{\ast},z^{\ast}))(x- x^{\ast})\right] + (C_{Y}\times (C_{Z}+\{z^{\ast}\})).
\end{equation*}
The proof of this theorem consists of several steps. First, we prove two important properties of this set $M$ and then we apply a separation theorem in order to obtain the multiplier rule. Finally, we show $t\neq 0_{Y^{\ast}}$ under the regularity assumption.
\par \textbf{(a)} We show that the nonempty set $M$ is convex. We prove the convexity for the translated set $M'=M-\{(0_{Y},z^{\ast})\}$ and immediately get the desired result. For this proof we fix two arbitrary pairs $(y_1,z_1),(y_2,z_2)\in M'$. Then there are elements $x_1,x_2 \in \hat{S}$ with
\begin{equation*}
    (y_i,z_i)\in D(F,G)(x^{\ast},(y^{\ast},z^{\ast}))(x_{i}- x^{\ast})+(C_Y\times C_Z), \quad i = 1,2
\end{equation*}
which equals to
\begin{equation*}
    D(F,G)(x^{\ast},(y^{\ast},z^{\ast}))(x_{i}- x^{\ast})\leq_{C_{Y}\times C_{Z}}(y_{i},z_{i}),\quad i = 1,2
\end{equation*}
resulting in
\begin{align*}
    (x_i-x^{\ast},(y_i,z_i)) &\in epi(D(F,G)(x^{\ast},(y^{\ast},z^{\ast})))\\
    &= T(epi(F,G),(x^{\ast},(y^{\ast},z^{\ast}))), \quad i = 1,2
\end{align*}
    
This contingent cone is convex because the map $(F,G)$ is cone-convex and, therefore, the epigraph $epi(F,G)$ is a convex set (see Thm 7.2.1). Then we obtain for all $\lambda \in [0,1]$
\begin{equation*}
    \lambda (x_1-x^{\ast},(y_1,z_1)) + (1-\lambda)(x_2-x^{\ast},(y_2,z_2))\in T(epi(F,G),(x^{\ast},(y^{\ast},z^{\ast}))),
\end{equation*}
implying
\begin{equation*}
    (\lambda y_1 + (1-\lambda)y_2,\lambda z_1 +(1-\lambda)z_2)\in D(F,G)(x^{\ast},(y^{\ast},z^{\ast}))(\lambda x_1+(1-\lambda)x_2-x^{\ast})+(C_{Y}\times C_{Z}).
\end{equation*}
Consequently, the set $M$ is convex.

\par \textbf{(b)} In the next step of the proof we show the equality
\begin{equation}
    M\bigcap \Big[ \left( -int(C_Y)\right) \times \left( -int(C_Z)\right) \Big] = \emptyset.
\end{equation}
Assume that this equality does not hold. Then there are elements $x\in \hat{S}$ and $(y,z)\in Y\times Z$ with
\begin{equation}
    \begin{aligned}
    (y,z+z^{\ast}) &\in \Big[ D(F,G)(x^{\ast},(y^{\ast},z^{\ast}))(x-x^{\ast}))+(C_Y\times (C_Z+\{z^{\ast}\})) \Big]\\
    & \cap \Big[ (-int(C_Y)\times (-int(C_Z))) \Big],
    \end{aligned}
\end{equation}
implying
\begin{equation*}
    (x-x^{\ast},(y,z))\in T(epi(F,G),(x^{\ast},(y^{\ast},z^{\ast}))).
\end{equation*}
This means that there are sequences $(x_n,(y_n,z_n))_{n\in \mathbb{N}}$ of elements in $epi(F,G)$ and a sequence $(\lambda_n)_{n\in \mathbb{N}}$ of positive real numbers with
\begin{equation*}
    (x^{\ast},(y^{\ast},z^{\ast})) = \lim_{n\to \infty}(x_n,(y_n,z_n))
\end{equation*}
and
\begin{equation}
    (x-x^{\ast},(y,z)) = \lim_{n\to \infty} \lambda_n (x_n-x^{\ast},(y_n-y^{\ast},z_n-z^{\ast})).
\end{equation}
Since $y\in -int(C_Y)$ by (38), we conclude $\lambda_n(y_n-y^{\ast})\in -int(C_Y)$ resulting in
\begin{equation}
    y_n\in y^{\ast}-int(C_Y)
\end{equation}
for sufficiently large $n\in \mathbb{N}$. Because of $(x_n,(y_n,z_n))\in epi(F,G)$ for all $n\in \mathbb{N}$ there are elements $\hat{y_n}\in F(x_n)$ with
\begin{equation*}
    y_n\in \{ \hat{y_n}\} + C_Y,\quad n\in \mathbb{N}.
\end{equation*}
Together with (40), for sufficiently large $n\in \mathbb{N}$ we obtain
\begin{equation*}
    \hat{y_n}\in \{ y^{\ast}\}-int(C_Y)-C_Y=\{ y^{\ast} \}-int(C_Y)
\end{equation*}
or
\begin{equation}
    (\{ y^{\ast} \}-int(C_Y))\cap F(x_n) \neq \emptyset
\end{equation}
for sufficiently large $n\in \mathbb{N}$. Moreover, from (38) we conclude $z+z^{\ast}\in -int(C_Z)$, and with (39) we obtain
\begin{equation*}
    \lambda_n(z_n-z^{\ast})+z^{\ast}\in -int(C_Z)
\end{equation*}
or
\begin{equation*}
    \lambda_n(z_n-(1-\frac{1}{\lambda_n})z^{\ast})\in -int(C_Z)
\end{equation*}
for sufficiently large $n\in \mathbb{N}$, implying 
\begin{equation}
    z_n-(1-\frac{1}{\lambda_n})z^{\ast}\in -int(C_Z).
\end{equation}
Since $y\neq 0_Y$ by (38), we conclude with (39) that $\lambda_n > 1$ for sufficiently large $n\in \mathbb{N}$. By assumption we have $z^{\ast}\in -C_Z$ and, therefore, we get from (42)
\begin{equation}
    z_n \in -C_Z-int(C_Z) = -int(C_Z).
\end{equation}
Because of $(x_n,(y_n,z_n))\in epi(F,G)$ for all $n\in \mathbb{N}$ there are elements $\hat{z_n}\in G(x_n)$ with
\begin{equation*}
    z_n\in {\hat{z_n}} + C_Z,\quad n\in \mathbb{N}.
\end{equation*}
Combined with (43),for sufficiently large $n\in \mathbb{N}$ we then get
\begin{equation*}
    \hat{z_n}\in {z_n}-C_Z\subseteq -int(C_Z)
\end{equation*}
and
\begin{equation}
    \hat{z_n}\in G(x_n)\cap (-C_Z).
\end{equation}
Hence, for sufficiently large $n\in \mathbb{N}$ we have $x_n\in \hat{S}$, $(\{ y^{\ast}\}-int(C_Y))\cap F(x_n) \neq \emptyset$ by (41), and $G(x_n)\cap (-C_Z)\neq \emptyset$ by (44) and therefore $(x^{\ast},y^{\ast})$ is not a weak minimizer of the problem $\mathrm{(P3)}$, which is a contradiction to the assumption of the theorem.

\par \textbf{(c)} In this step we now prove the first part of the theorem. By part (a) the set $M$ is convex and by (b) and equality (37) holds. By convex sets separation theorem, there are continuous linear functionals $t\in Y^{\ast}$ and $u\in Z^{\ast}$ with $(t,u)\neq (0_{Y^{\ast}},0_{Z^{\ast}})$ and a real number $\gamma >0$ so that
\begin{equation}
    t(c_Y)+u(c_Z)<\gamma \leq t(y)+u(z),\quad \forall c_Y\in -int(C_Y),\:c_Z\in -int(C_Z),\:(y,z)\in M.
\end{equation}
Since $(0,z^{\ast})\in M$, we obtain from (45) for $c_Y = 0_Y$
\begin{equation}
    u(c_Z)<u(z^{\ast}),\quad \forall c_Z\in -int(C_Z).
\end{equation}
If we assume that $u(c_Z)>0$ for a $c_Z\in -int(C_Z)$, we get a contradiction to (46) because $C_Z$ is a cone. Therefore, we obtain the fact that
\begin{equation*}
    u(c_Z)\leq 0,\quad \forall c_Z\in -int(C_Z),
\end{equation*}
resulting in $u\in C_{Z^{\ast}}$ because $C_Z\subseteq cl(int(C_Z))$. For $(0,z^{\ast})\in M$ and $c_Z = 0_Z$ we get from (45)
\begin{equation}
    t(c_Y)<u(z^{\ast})\leq 0,\quad \forall c_Y\in -int(C_Y)
\end{equation}
(notice that $z^{\ast}\in -C_Z$ and $u\in C_{Z^{\ast}}$). This inequality implies $t\in C_{Y^{\ast}}$. From (46) and (47) we immediately obtain $u(z^{\ast})=0$. In order to prove the inequality of the multiplier rule we conclude from (45) with $c_Y = 0_Y$ and $c_Z = 0_Z$
\begin{equation*}
    t(y)+u(z)\geq 0,\quad \forall (y,z) = D(F,G)(x^{\ast},(y^{\ast},z^{\ast}))(x-x^{\ast})\:\,with\:\, x\in \hat{S}.
\end{equation*}
Hence, the first part of the theorem is shown.

\par \textbf{(d)} Finally, we prove $t\neq 0_{Y^{\ast}}$ under the regularity assumption (36). For an arbitrary $\hat{z}\in Z$ there are elements $x\in \hat{S}, c_Z\in C_Z$ and non-negative real numbers $\alpha$ and $\beta$ with
\begin{equation*}
    \hat{z} = z + \beta(c_Z+z^{\ast}) \:for\: (y,z) = D(F,G)(x^{\ast},(y^{\ast},z^{\ast}))(\alpha(x-x^{\ast})).
\end{equation*}
Since $D(F,G)(x^{\ast},(y^{\ast},z^{\ast}))$ is positively homogeneous by Thm 7.2.2, (notice that we do not need the cone-convexity for this proof), we can write
\begin{equation*}
    (y,z) = \alpha D(F,G)(x^{\ast},(y^{\ast},z^{\ast}))(x-x^{\ast}):=\alpha(\Tilde{y},\Tilde{z}).
\end{equation*}
Assume that $t=0_{Y^{\ast}}$. Then we conclude from the multiplier rule
\begin{equation*}
    \begin{aligned}
    u(\hat{z}) &= u(z)+\beta u(c_Z+z^{\ast})\\
               &= \alpha u(\Tilde{z}) + \beta u(c_Z) + \beta u(z^{\ast})\\
               &= \alpha (u(\Tilde{z})+t(\Tilde{y})) + \beta u(c_Z) + \beta u(z^{\ast})\\
               &\geq 0.
    \end{aligned}
\end{equation*}
Because $\hat{z}$ is arbitrary chosen we have
\begin{equation*}
    u(\hat{z})\geq 0,\quad \forall \hat{z}\in Z,
\end{equation*}
implying $u = 0_{Z^{\ast}}$. But this is a contradiction to $(t,u)\neq (0_{Y^{\ast}},0_{Z^{\ast}})$.
\end{proof}

\begin{framed}
\noindent \emph{Remark:} This theorem extends the Lagrange multiplier rule as a necessary optimality condition to set-valued optimization. It also extends the so-called Karush-Kuhn-Tucker condition if $t\neq 0_{Y^{\ast}}$. Since a minimizer of the problem $\mathrm{(P3)}$ is also a weak minimizer, this multiplier rule is a necessary optimality condition for a minimizer as well. Besides, the regularity condition extends  the concept of constraint qualifications to set-valued optimization. 
\end{framed}

\section{On the Limits of the Lagrange Multiplier Rule}
In this section, we will first show two examples where the Lagrangian multiplier rule fails, then we will analyze these examples in detail and give a more precise description of classical Karush-Kuhn-Tucker theorem. This part is mainly refer to Luis A.Fernandez's paper \cite{fernandez1997}.

\subsection{The Failure of Lagrangian Methods}
As in \textsection3.1, one is tempted to summarize the Karush-Kuhn-Tucker theorem (as did the well-known expert Ioffe in \cite{ioffe1993lagrange}) by saying that, ``... in problems with finite equality constraints, the Lagrange multiplier rule is valid under the assumption that the cost functions and constraint functions are only Fr$\Acute{e}$chet differentiable at the solution."
\par However, actually one cannot remove the hypothesis of continuity of the constrain functions in a neighborhood of the solution. As we will see the following two examples below.
\begin{framed}
\noindent \emph{\textbf{Example 1.}} Let's consider the functions $f_{0},f_{1}: \mathbb{R}^n\to \mathbb{R}$ defined by $f_{0}(x,y) = x$ and
\begin{equation*}
    f_{1}(x,y) = \begin{cases}
    y & \text{if } x \geq 0;\\
    y-x^2 & \text{if } x<0,\:y\leq 0;\\
    y+x^2 & \text{if } x<0,\:y>0.
    \end{cases}
\end{equation*}
We consider the following optimization problem:
\begin{equation*}
\begin{aligned}
\min \quad & f_{0}(x,y)\\
\textrm{s.t.} \quad & f_{1}(x,y) = 0.\\
\end{aligned}
\end{equation*}
\end{framed}

\noindent \emph{\textbf{Solution 1.}} Evidently $(0,0)$ is the unique solution of the problem above. It's easy to calculate that $\nabla f_{0}(0,0) = (1,0)$ and $\nabla f_{1}(0,0) = (0,1)$. If the multiplier rule were valid in this situation, then there would exists some nonzero $(\lambda_{0},\lambda_{1})\in \mathbb{R}^2$ such that
\begin{equation*}
    \lambda_{0} \nabla f_{0}(0,0) + \lambda_{1} \nabla f_{1}(0,0) = (\lambda_{0},0) + (0,\lambda_{1}) = (0,0),   
\end{equation*}
which contradicts our assumption.

\begin{framed}
\noindent \emph{\textbf{Example 2.}} Let's consider the functions $f_{0},f_{1}:\mathbb{R}^n\to \mathbb{R}$ defined by $f_{0}(x,y)=y$ and \begin{equation*}
    f_{1}(x,y) = \begin{cases}
    x & \text{if } y\geq 0;\\
    e^x+y^2-1 & \text{if } y<0,\:x\geq 0;\\
    e^x-y^2-1 & \text{if } y<0,\:x<0.
    \end{cases}
\end{equation*}
We consider the following optimization problem:
\begin{equation*}
\begin{aligned}
\min \quad & f_{0}(x,y)\\
\textrm{s.t.} \quad & f_{1}(x,y) = 0.\\
\end{aligned}
\end{equation*}
\end{framed}

\noindent \emph{\textbf{Solution 2.}} Evidently $(0,0)$ is the unique solution of the problem above. It's easy to calculate that $\nabla f_{0}(0,0) = (0,1)$ and $\nabla f_{1}(0,0) = (1,0)$. If the multiplier rule were valid in this situation, then there would exists some nonzero $(\lambda_{0},\lambda_{1})\in \mathbb{R}^2$ such that
\begin{equation*}
    \lambda_{0} \nabla f_{0}(0,0) + \lambda_{1} \nabla f_{1}(0,0) = (0,\lambda_{0}) + (\lambda_{1},0) = (0,0),
\end{equation*}
which contradicts our assumption.
\begin{framed}
\noindent \emph{Remark:} The above two examples really confuse us. It's obvious to verify that $f_{0}(x)$ and $f_{1}(x)$ are Fr$\Acute{e}$chet differentiable at $(0,0)$, but why the Lagrange multiplier rule fails? We will make further explanation in \textsection 8.2.
\end{framed}

\subsection{More Precise Claim}
In Luis A.Fernandez's paper \cite{fernandez1997}, the author makes the following claim, emphasizing the importance of local continuity.
~\par ~
\par \noindent \emph{\textbf{Claim.}} If the function defining one of the equality constraints is Fr$\Acute{e}$chet differentiable at the solution and discontinuous in every neighborhood of the solution, then the Lagrange multiplier rule can fail.
\par Now we can easily verify the discontinuity of constraint functions in the above two optimization problems respectively by using straightforward arguments. 
\begin{framed}
\noindent \emph{\textbf{Verification Example 1.}}
\par (1) $f_{0}$ is continuously Fr$\Acute{e}$chet differentiable in $\mathbb{R}^2$;
\par (2) $f_{1}$ is continuous function in $\mathbb{R}^2\setminus \{(x,0):x<0\}$; in particular, $f_{1}$ is discontinuous in every neighborhood of $(0,0)$;
\par (3) $f_{1}$ is Fr$\Acute{e}$chet differentiable at $(0,0)$.
~\par ~
\par \noindent \emph{\textbf{Verification Example 2.}}
\par (1) $f_{0}$ is continuously Fr$\Acute{e}$chet differentiable in $\mathbb{R}^2$;
\par (2) $f_{1}$ is continuous function in $\mathbb{R}^2\setminus \{(y,0):y<0\}$; in particular, $f_{1}$ is discontinuous in every neighborhood of $(0,0)$;
\par (3) $f_{1}$ is Fr$\Acute{e}$chet differentiable at $(0,0)$.
\end{framed}
According to the verification above, we know the equality constraints $f_{1}$ are Fr$\Acute{e}$chet differentiable at the optimal solution but not continuous in any neighborhood of the optimal solution, thus the Lagrange multiplier rule can fail. The following more precise result concerning Lagrange multiplier rule is presented, belonging to Halkin \cite{halkin1974implicit}.

\begin{theorem}
\textbf{(Fritz John Condition)} Let $x^{\ast}$ be a local optimal solution of problem $\mathrm{(P1)}$, objective function $f:\mathbb{R}^n \mapsto \mathbb{R}$, constraint functions $g_{i}:\mathbb{R}^n \mapsto \mathbb{R}(i = 1,\dots,m)$ and $h_{j}:\mathbb{R}^n \mapsto \mathbb{R}(j = 1,\dots,n)$ are all Fr$\Acute{e}$chet differentiable at $x^{\ast}$. Constraint functions $h_{j}:\mathbb{R}^n \mapsto \mathbb{R}(j = 1,\dots,n)$ are all continuous in a neighborhood of $x^{\ast}$. Then there exist $\overline{\lambda_{0}}, \overline{\lambda_{1}}, \dots, \overline{\lambda_{m}}$ and $\overline{\mu_{1}},\dots,\overline{\mu_{n}}$ satisfying
\begin{equation*}
\begin{aligned}
&\overline{\lambda_{0}}\nabla f(x^{\ast}) + \sum_{i=1}^m \overline{\lambda_{i}} \nabla g_{i}(x^{\ast}) + \sum_{j=1}^n \overline{\mu_{j}} \nabla h_{j}(x^{\ast}) = 0,\\
&\overline{\lambda_{i}} \geq 0,\: \overline{\lambda_{i}}g_{i}(x^{\ast}) = 0,\: i = 1,\dots,m
\end{aligned}
\end{equation*}
\end{theorem}

Finally, let us mention that in each previous example there is a function which is neither convex nor locally Lipschitz; hence it is not possible to apply a generalized Lagrange multiplier rule via different subdifferentials for the corresponding mathematical programming problems.

\newpage
\bibliography{ref}
\newpage
}
\end{spacing}

\end{document}